\documentclass[12pt]{amsart}
\usepackage[utf8]{inputenc}
\usepackage{amsmath,amssymb,amsrefs}
\usepackage{mleftright}
\usepackage{mathtools}
\usepackage[margin=2.5cm]{geometry}

\usepackage{tikz}
\usepackage{hyperref}
\hypersetup{colorlinks=true, linkcolor=blue, citecolor=magenta, filecolor=magenta, urlcolor=magenta}

\newtheorem{theorem}{Theorem}[section]
\newtheorem{lemma}[theorem]{Lemma}
\newtheorem{proposition}[theorem]{Proposition}
\newtheorem{corollary}[theorem]{Corollary}

\theoremstyle{definition}
\newtheorem{definition}[theorem]{Definition}

\newtheorem{example}[theorem]{Example}
\newtheorem{question}[theorem]{Question}
\newtheorem{remark}[theorem]{Remark}

\newcommand{\NN}{{\mathbb{N}}}
\newcommand{\RR}{{\mathbb{R}}}
\newcommand{\ZZ}{{\mathbb{Z}}}
\newcommand{\lcm}{\mathrm{lcm}}

\newcommand{\ceil}[1]{\mleft\lceil{#1}\mright\rceil}
\newcommand{\floor}[1]{\mleft\lfloor{#1}\mright\rfloor}
\newcommand{\fractional}[1]{\mleft\{{#1}\mright\}}

\DeclareMathOperator{\age}{age}

\makeatletter
\newcommand\cdotfill{%
    \leavevmode\cleaders\hb@xt@.44em{\hss$\cdot$\hss}\hfill\kern\z@
}
\makeatother

\title{Local \(h^*\)-polynomials for one-row Hermite normal form simplices}

\author[E.~Bajo]{Esme Bajo}
\address{University of California\\ Berkeley, CA, USA}
\email{esme@berkeley.edu}
\author[B.~Braun]{Benjamin Braun}
\address{Department of Mathematics\\
         University of Kentucky\\
         Lexington, KY 40506--0027}
\email{benjamin.braun@uky.edu}
\author[G.~Codenotti]{Giulia Codenotti}
\address{Freie Universit\"at Berlin\\ Arnimallee 3, 14105 Berlin, Germany}
\email{giulia.codenotti@fu-berlin.de}
\author[J.~Hofscheier]{Johannes Hofscheier}
\address{School of Mathematical Sciences\\University of Nottingham\\ Nottingham, NG7 2RD, UK}
\email{johannes.hofscheier@nottingham.ac.uk}
\author[A.~R.~Vindas-Mel\'endez]{Andr\'es R. Vindas-Mel\'endez}
\address{Department of Mathematics\\ Harvey Mudd College}
\email{avindasmelendez@g.hmc.edu}

\date{2 January 2025}

\begin{document}

\begin{abstract}
  The local \(h^*\)-polynomial of a lattice polytope is an important invariant arising in Ehrhart theory.
  Our focus is on lattice simplices presented in Hermite normal form with a single non-trivial row.
  We prove that when the off-diagonal entries are fixed, the distribution of coefficients for the local \(h^*\)-polynomial of these simplices has a limit as the normalized volume goes to infinity.
  Further, this limiting distribution is determined by the coefficients for a particular choice of normalized volume.
  We also provide an analysis of two specific families of such simplices to illustrate and motivate our main result.
\end{abstract}

\maketitle

\section{Introduction}

\subsection{Background}
Given a lattice polytope \(P\), i.e., a convex polytope with vertices in the integer lattice, there are two important invariants related to counting the number of integer points in integer dilates of \(P\).
First, the Ehrhart \(h^*\)-polynomial is a well-studied object with connections to commutative algebra, algebraic geometry, enumerative combinatorics, and other fields.
Second, recently there has been renewed interest in the study of the \emph{local $h^*$-polynomial} for \(P\), which is more complicated to define and has not been as extensively studied.
This latter invariant is the primary focus of this work.
We denote the local \(h^*\)-polynomial by \(B(P;z)\).

The local $h^*$-polynomial has arisen in multiple contexts using different notation.
For a detailed survey regarding local $h^*$-polynomials, see Section~2 of the recent paper by Borger, Kretschmer, and Nill~\cite{borgerkretschmernill}.
The local $h^*$-polynomial was defined in substantial generality by Stanley in~\cite{StanleySubdivisions}*{Example 7.13}, extending work first presented by Betke and McMullen~\cite{BetkeMcMullen}.
Local \(h^*\)-polynomials were also studied by Borisov and Mavlyutov in connection to Calabi-Yau complete intersections in Gorenstein toric Fano varieties, where they were referred to as $\Tilde{S}$-polynomials~\cite{BorisovMavlyutov}.
Local $h^*$-polynomials for simplices are sometimes referred to as \emph{box polynomials}; these were studied by Gelfand, Kapranov, and Zelevinsky~\cite{GKZ91}, who identified the importance of lattice simplices with vanishing local \(h^*\)-polynomials.
Lattice polytopes with vanishing local \(h^*\)-polynomials are called \emph{thin polytopes}, and these have recently been further investigated by Borger, Kretschmer, and Nill~\cite{borgerkretschmernill}.
Nill and Schepers ~\cites{nillschepers,borgerkretschmernill} observed that any lattice polytope admitting a regular unimodular triangulation has a unimodal local \(h^*\)-polynomial.

Unimodality of local \(h^*\)-polynomials also play a role in the study of unimodality for \(h^*\)-polynomials.
Schepers and Van Langenhoven~\cite{svl} introduced the concept of a \emph{box unimodal} triangulation, which is a lattice triangulation \(T\) of a lattice polytope for which every face of \(T\) has a unimodal local \(h^*\)-polynomial.
The motivation for the term ``box unimodal'' is the ``box polynomial'' nomenclature used by Gelfand, Kapronov, and Zelevinsky.
Schepers and Van Langenhoven proved that if a reflexive lattice polytope has a box unimodal triangulation, then it has a unimodal \(h^*\)-polynomial. This generated interest in determining which simplices have unimodal local \(h^*\)-polynomials, since these are the simplices appearing in box unimodal triangulations.
As one example of results in this direction, Solus and Gustafsson proved that every \(s\)-lecture hall order polytope admits a box unimodal triangulation~\cite{solusgustafsson}.

\subsection{Lattice simplices and our contributions}
Motivated by the above context, our goal is to investigate unimodality of local \(h^*\)-polynomials for lattice simplices. 
Rather than proving unimodality results for simplices which have a strong property, such as having a regular unimodular triangulation or the integer decomposition property, we focus on a family of simplices with arithmetic structure that is rich enough to capture a wide range of behaviors.
Within this family, we investigate the presence of, and global behavior for, local \(h^*\)-unimodality.
Thus, our goal is to develop a better idea of what ``typical'' behavior within a family of simplices means, and how to approach establishing unimodality results in this context using elementary methods.

The local \(h^*\)-polynomial of lattice simplices has the following beautiful geometric interpretation, which for the rest of this paper we take as the definition.
Let \(\{v_1,\ldots,v_{d+1}\}\) be the vertices of a lattice simplex \(S\), and let
\[
\Pi_S\coloneqq \left\{\sum_i \lambda_i(1,v_i):0< \lambda_i<1\right\}
\]
define the open parallelepiped for \(\{1\}\times S\).
Then the local \(h^*\)-polynomial is 
\[
B(S;z):=\sum_{(m_0,\dots, m_{d})\in \Pi_S}z^{m_0} \, ,
\]
i.e., \(B(S;z)\) encodes the distribution of lattice points through the open parallelepiped of \(\{1\}\times S\) with respect to the \(0\)-th coordinate.
We refer to the \(0\)-th coordinate as the \emph{height} of the point.
To emphasize this distributional perspective, and to allow us to compare the shape of coefficient vectors of different local \(h^*\)-polynomials, we study the coefficients of 
\[
B(S;z)/B(S;1),
\]
which encode the probability distribution for lattice points in \(\Pi_S\) with respect to height.

From this perspective, it seems natural that the local \(h^*\)-polynomial might have unimodal coefficients, as the parallelepiped is ``fatter'' geometrically in the middle than on the ends.
It seems plausible that unimodality might be expected in this setting, even without assumptions, such as admitting a regular unimodular triangulation or the integer decomposition property.
However, as we will see in this work, it is not clear whether or not these intuitions are correct; for example, Figure~\ref{fig:11sample505} and Figure~\ref{fig:fractionunimodalbypartition} suggest a variety of possible conjectures.

Because an arbitrary lattice simplex is arithmetically complicated, we restrict our attention to a set of simplices with a more manageable arithmetical structure.
Lattice simplices are classified through their Hermite normal form; see Theorem~\ref{thm:hermite} for the precise statement.
In Section~\ref{sec:background}, we define Hermite normal form simplices and recall how to compute their \(h^*\)- and local \(h^*\)-polynomials.
The family of simplices that we study are one-row Hermite normal form simplices, which are those unimodularly equivalent to the convex hull of the rows of an integer matrix as in~\eqref{eq:onerowintro}, specified by parameters $a_1, \dots, a_{d-1}, N$ with \(0\leq a_i< N\) for all \(i\). Note that the normalized volume of these simplices is exactly the parameter $N$.

\begin{equation}\label{eq:onerowintro}
H = 
\begin{bmatrix}
		 0 &  0 &  0 & \cdots &  0 &  0 &  0\\
		 1 &  0 &  0 & \cdots &  0 &  0 &  0\\
		 0 &  1 &  0 & \cdots &  0 &  0 &  0\\
		0 &  0 &  1 & \cdots &  0 &  0 &  0\\
		\vdots & \vdots & \vdots & \ddots & \vdots & \vdots & \vdots \\
		 0 &  0 &  0 & \cdots &  1 &  0 &  0\\
		 0 &  0 &  0 & \cdots &  0 &  1 &  0\\
		 a_1 & a_2 & a_3 & \cdots & a_{d-2} & a_{d-1} & N
	\end{bmatrix}.
\end{equation}

After providing detailed background regarding Hermite normal form, local \(h^*\)-polynomials, and the relationship between local \(h^*\)-polynomials and Stapledon \(a/b\)-decompositions for \(h^*\)-polynomials, our first contribution is to investigate two special subfamilies of one-row Hermite form simplices that exhibit distinct behavior with regard to their local \(h^*\)-polynomials. 
These special families illustrate both the variety of behavior observed with local \(h^*\)-polynomials and the proof techniques that we use throughout the paper.
This is the content of Section~\ref{sec:examples}, where we provide a complete investigation of local \(h^*\)-polynomials for ``all-ones'' simplices ($a_i =1$ for all $i$) and for ``geometric sequence'' simplices ($a_i=q^{d-i}$ for all $i$).
In the all-ones case, we find that their local \(h^*\)-polynomials are either constant or nearly so, as exemplified by Figure~\ref{fig:allonesdist}.
We characterize the all-ones simplices with unimodal local \(h^*\)-polynomials.
For the geometric sequence simplices, we find that their local \(h^*\)-polynomials have a pronounced unimodal behavior, as exemplified by Figure~\ref{fig:q3_k12_geometricdist}.
We prove that all the geometric sequence simplices have unimodal local \(h^*\)-polynomials, and further show that they do not have the integer decomposition property and thus do not fall within the scope of prior work~\cites{nillschepers,borgerkretschmernill}.

\begin{figure}[ht]
\centering
\includegraphics[width=0.6\textwidth]{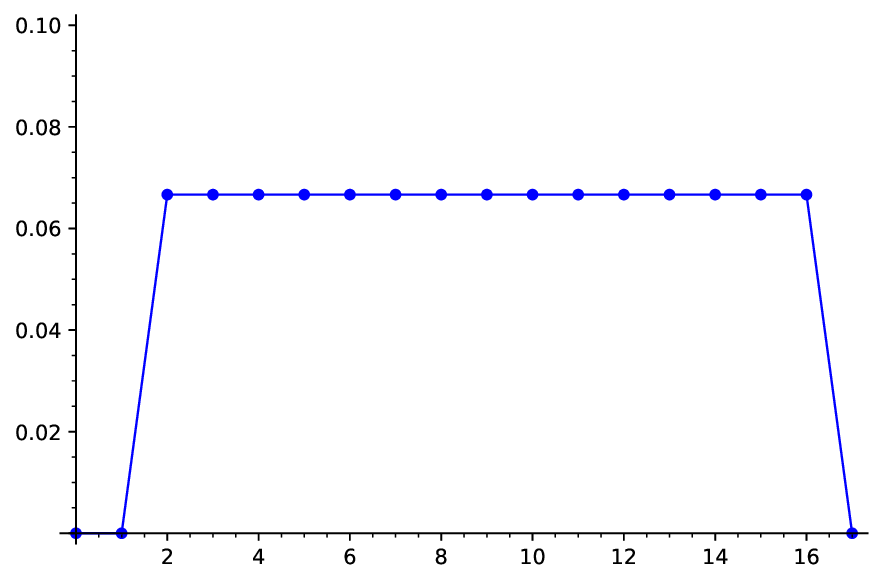}
\caption{The distribution of the coefficients of \(B(S;z)/B(S;1)\) for the one-row Hermite normal form simplex \(S\) with non-trivial row \( (1,1,\ldots,1,331) \) in dimension \(17\). Note that \(331=22\cdot 15+1\) and \(B(S;z)=22\cdot\sum_{i=2}^{16}z^i\).}
\label{fig:allonesdist}
\end{figure}

\begin{figure}[ht]
\centering
\includegraphics[width=0.6\textwidth]{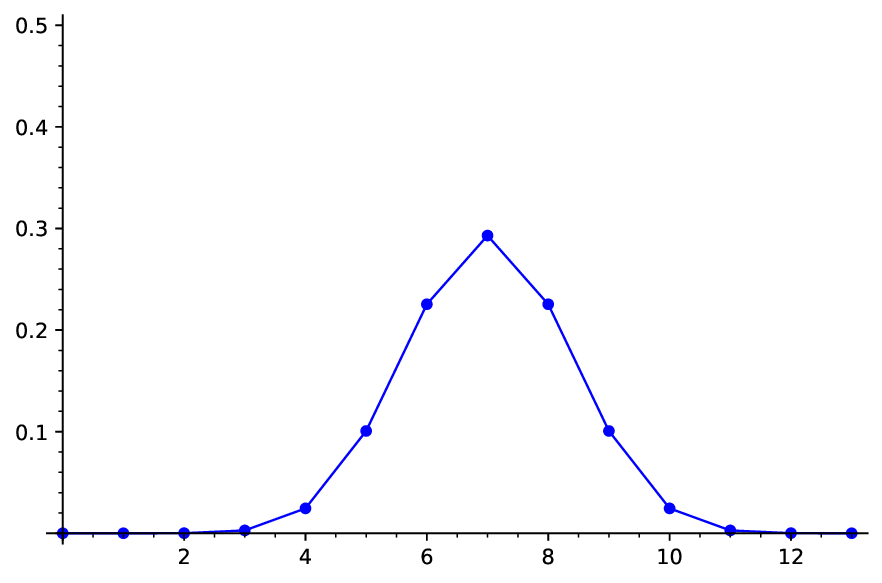}
\caption{The distribution of the coefficients of $B(S;z)/B(S;1)$\\ for the one-row Hermite normal form simplex $S$ with non-trivial \\row $(3^{11},3^{10},3^9,\ldots,3^2,3,1,3^{12})$ in dimension $13$.}
\label{fig:q3_k12_geometricdist}
\end{figure}

To study the set of all one-row Hermite normal form simplices, there are several ways to proceed.
One approach, which we do not take in this work, is to fix \(N\) and vary the values \(a_1,\ldots,a_{d-1}\).
Limited experimental data suggests that this process often results in simplices that do not have the integer decomposition property, but which do have unimodal local \(h^*\)-polynomials.
For example, in a random sample of \(100\) simplices of dimension \(11\) with \(N=505\), none of these simplices have the integer decomposition property, but all of them have unimodal local \(h^*\)-polynomials.
The distributions for these polynomials are plotted in Figure~\ref{fig:11sample505}, which demonstrates that the integer decomposition property does not fully explain local \(h^*\)-unimodality.

\begin{figure}
\centering
\includegraphics[width=0.7
\textwidth]{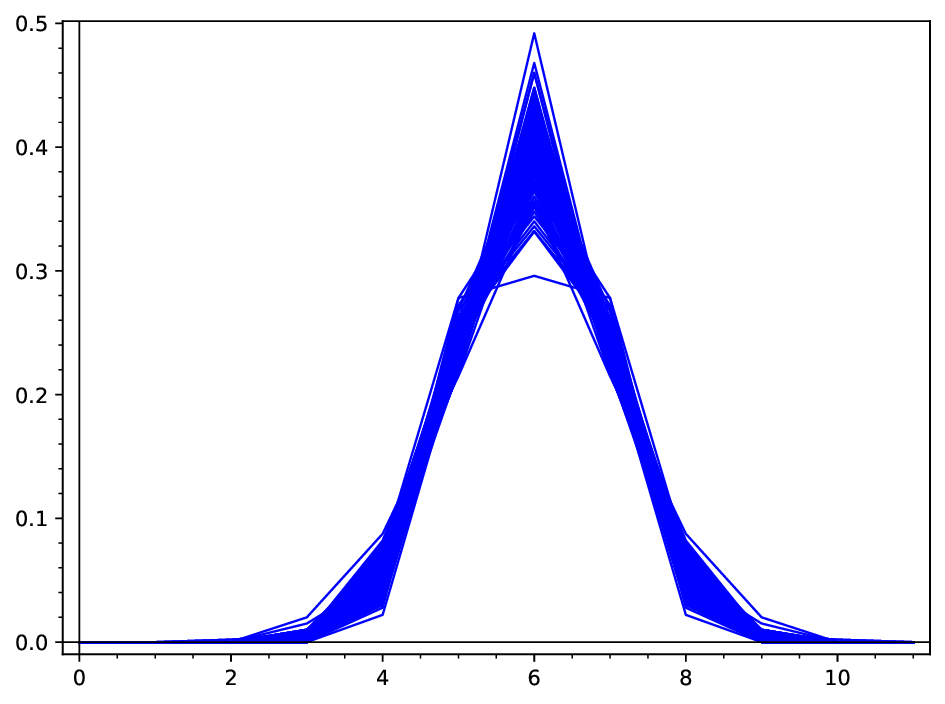}
\caption{The distributions for \(100\) local \(h^*\)-polynomials of \(11\)-dimensional one-row Hermite normal form simplices with \(N=505\).
None of these simplices have the integer decomposition property.}
\label{fig:11sample505}
\end{figure}

Another approach, which is the focus of our main result, is to consider one-row Hermite normal form simplices where the parameters \(a_i\) are fixed and then analyze the asymptotic behavior of \(B(S;z)/B(S;1)\) as the normalized volume \(N\to \infty\).
In Section~\ref{sec:limits}, we show that this asymptotic behavior is determined by the distribution of the coefficients for a particular normalized volume value deterimed by the values \(a_1,\ldots,a_{d-1}\).
This demonstrates that in the one-row case, the off-diagonal elements of a Hermite normal form matrix have a stronger influence on the local \(h^*\)-polynomial than the normalized volume.
Our main result is Theorem~\ref{thm:fullasymptotic}, which can be summarized as follows;
\begin{theorem}\label{thm:simple}
Fix \(a_1, \dots, a_{d-1} \in \ZZ_{\geq 1}\) and let \(M \coloneqq \mathrm{lcm}(a_1, \dots, a_{d-1}, -1+\sum_{i=1}^{d-1}a_i)\).
Let \(S_N\) denote the one row Hermite normal form simplex for these values of \(a_1, \dots, a_{d-1}\) and normalized volume \(N\).
As \(N\to\infty\), the distributions for the local \(h^*\)-polynomials of \(S_N\) converge to the distribution obtained when \(N=M+1\), i.e.,
\[
		\lim_{N \to \infty}B(S_{N};z)/B(S_{N};1) = B(S_{M+1};z)/B(S_{M+1};1) \, .
    \]
\end{theorem}

The results in this paper are rather technical and thus the details of our definitions and arguments vary somewhat from the exposition provided so far.
Since the proofs in Section~\ref{sec:limits} rely on several technical lemmas regarding floor and ceiling functions, these lemmas are given in Section~\ref{sec:floor}.
We conclude our paper in Section~\ref{sec:conclusion} with several further questions motivated by this work.
All computations in this article were done using SageMath~\cite{sage}.

\section{Properties of Local \texorpdfstring{$h^*$}--Polynomials}\label{sec:background}

\subsection{Hermite normal form simplices}

For every lattice simplex $S$, there exists a unique matrix $H$ representing the vertices of $S$ up to unimodular equivalence \cite{schrijver}, called the \emph{Hermite normal form} of $S$. 

\begin{theorem}\label{thm:hermite}
Every $d$-dimensional lattice simplex in $\RR^d$ is unimodularly equivalent to a simplex $S$ arising as the convex hull of the rows of a $(d+1)\times d$ integer matrix $H$ of the following form:
\begin{itemize}
    \item $a_{0,i}=0$ for $i=1,\ldots,d$
    \item $a_{i,i}\in \ZZ_{\geq 1}$ for $i=1,\ldots,d$
    \item $0\leq a_{i,j}<a_{i,i}$ when $j<i$ for $i=1,\ldots,d$
    \item $a_{i,j}=0$ for $j>i$ for $i=1,\ldots,d$.
\end{itemize}
\end{theorem}

We are using the convention that the \emph{rows of $H$ denote the vertices of $S$}, whereas some authors use column form.
We denote by $A$ the matrix $H$ with an initial column of ones appended.
This is equivalent to lifting the configuration of rows of $H$ to height one in a space one dimension higher.

\begin{example}\label{ex:hnfgeneral}
An example of a Hermite normal form $H$ of a simplex and the extended matrix $A$: 
\[
H = 
\left[
\begin{array}{ccccc}
0 & 0 & 0 & 0 & 0 \\
3 & 0 & 0 & 0 & 0 \\
2 & 5 & 0 & 0 & 0 \\
1 & 0 & 2 & 0 & 0 \\
0 & 0 & 0 & 1 & 0 \\
8 & 8 & 2 & 6 & 9
\end{array}\right],  \,\,\, A = \left[
\begin{array}{cccccc}
1 & 0 & 0 & 0 & 0 & 0 \\
1 & 3 & 0 & 0 & 0 & 0 \\
1 & 2 & 5 & 0 & 0 & 0 \\
1 & 1 & 0 & 2 & 0 & 0 \\
1 & 0 & 0 & 0 & 1 & 0 \\
1 & 8 & 8 & 2 & 6 & 9
\end{array}\right]\,.
\]
\end{example}

We focus on the following special class of Hermite normal form matrices, which have been extensively studied \cites{braundavisanti, braundavisintclosed,braundavissolus,braunhanely, braunliu,conrads,hibihermite,hibitsuchiyayoshida,higashitanicounterexamples,higashitaniprime,liusolus,payne,soluslocal,solusnumeral,tsuchiyagorenstein}.

\begin{definition} 
We say that a simplex $S$ is of \emph{one-row Hermite normal form} if, given its Hermite normal form $H$ in the notation above, we have $a_{i,i}=1$ for all $i=1,\ldots,d-1$.
We often refer to such a simplex by the values in the $(d+1)$-st row,
\[
	(a_1,a_2,\ldots,a_{d-1},N),
\]
where we write $a_i$ for $a_{d,i}$ and $N$ for $a_{d,d}$, as shown in~\eqref{eq:onerow}:
\end{definition}

\begin{equation}\label{eq:onerow}
A = 
\begin{bmatrix}
		1 &  0 &  0 &  0 & \cdots &  0 &  0 &  0\\
		1 &  1 &  0 &  0 & \cdots &  0 &  0 &  0\\
		1 &  0 &  1 &  0 & \cdots &  0 &  0 &  0\\
		1 &  0 &  0 &  1 & \cdots &  0 &  0 &  0\\
		\vdots & \vdots & \vdots & \vdots & \ddots & \vdots & \vdots & \vdots \\
		1 &  0 &  0 &  0 & \cdots &  1 &  0 &  0\\
		1 &  0 &  0 &  0 & \cdots &  0 &  1 &  0\\
		1 & a_1 & a_2 & a_3 & \cdots & a_{d-2} & a_{d-1} & N
	\end{bmatrix}.
\end{equation}

\begin{example}\label{ex:hnf6}
An example of a simplex $S$ in one-row Hermite normal form is given by
\[
H = 
\left[
\begin{array}{ccccc}
0 & 0 & 0 & 0 & 0 \\
1 & 0 & 0 & 0 & 0 \\
0 & 1 & 0 & 0 & 0 \\
0 & 0 & 1 & 0 & 0 \\
0 & 0 & 0 & 1 & 0 \\
1 & 1 & 1 & 1 & 6
\end{array}\right]\,.
\]
\end{example}

It is a short exercise to determine the volume of $S$ from the matrix $A$.

\begin{proposition}
The normalized volume of a simplex given in Hermite normal form is $\prod_{i=1}^da_{i,i}$, for $a_{i,i}\in A$.
Thus, for a one-row Hermite normal form simplex, the normalized volume is $N$.
\end{proposition}

\subsection{\texorpdfstring{$h^*$}-- and local \texorpdfstring{$h^*$}--polynomials}\label{sec:box-h-star-polynomials}

In this section, we define the local $h^*$-polynomial, not in terms of height, but instead using an equivalent definition involving the \emph{age} of a point.
We use this alternative definition because it supports cleaner technical arguments of our main results.

The arithmetic structure of the lattice points in the cone over $S$ is captured by the lattice generated by $A^{-1}$, which for a one-row Hermite normal form simplex with non-trivial row given by $(a_1,a_2,\ldots,a_{d-1},N)$ is

\begin{equation}\label{eq:onerowinverse}
A^{-1} = 
\begin{bmatrix}
		\hphantom{-}1 & \hphantom{-}0 & \hphantom{-}0 & \hphantom{-}0 & \cdots & \hphantom{-}0 & \hphantom{-}0 & \hphantom{-}0\\
		-1 & \hphantom{-}1 & \hphantom{-}0 & \hphantom{-}0 & \cdots & \hphantom{-}0 & \hphantom{-}0 & \hphantom{-}0\\
		-1 & \hphantom{-}0 & \hphantom{-}1 & \hphantom{-}0 & \cdots & \hphantom{-}0 & \hphantom{-}0 & \hphantom{-}0\\
		-1 & \hphantom{-}0 & \hphantom{-}0 & \hphantom{-}1 & \cdots & \hphantom{-}0 & \hphantom{-}0 & \hphantom{-}0\\
		\vdots & \vdots & \vdots & \vdots & \ddots & \vdots & \vdots & \vdots \\		-1 & \hphantom{-}0 & \hphantom{-}0 & \hphantom{-}0 & \cdots & \hphantom{-}1 & \hphantom{-}0 & \hphantom{-}0\\
		-1 & \hphantom{-}0 & \hphantom{-}0 & \hphantom{-}0 & \cdots & \hphantom{-}0 & \hphantom{-}1 & \hphantom{-}0\\
		\frac{-1+\sum_{i=1}^{d-1}a_i}N & -\frac{a_1}N & -\frac{a_2}N & -\frac{a_3}N & \cdots & -\frac{a_{d-2}}N & -\frac{a_{d-1}}N & \frac1N
	\end{bmatrix}.
\end{equation}

\begin{definition}
We define the \emph{lattice for $S$} as \[
\Lambda=\Lambda(S)\coloneqq \ZZ^{d+1}\cdot A^{-1}
\]
and the \emph{parallelepiped group for $S$} as 
\[
\Gamma=\Gamma(S):=\Lambda(S)/\ZZ^{d+1} \, .
\]
Define the \emph{age} of an element $x=(x_0, \dots, x_d)\in \Lambda(S)$ as 
\[
\age(x) \coloneqq \sum_{i=0}^d \fractional{x_i} \, .
\]
Here \(\{x\}\) denotes the \emph{fractional part} of a real number \(x \in \mathbb{R}\), i.e., \(\{x\} = x - \lfloor x \rfloor\) where \(\lfloor x\rfloor\) denotes the largest integer less than or equal to \(x\).
\end{definition}

Note that when \(S\) is a one-row Hermite normal form simplex, the elements of \(\Gamma\) are parameterized for \(0\leq \ell\leq N-1\) by
\[
    \left(\frac{\ell(-1+\sum_{i=1}^{d-1}a_i)}{N} , -\frac{\ell a_1}{N} , -\frac{\ell a_2}{N} , -\frac{\ell a_3}{N} , \ldots ,-\frac{\ell a_{d-2}}{N} ,-\frac{\ell a_{d-1}}{N} , \frac{\ell}{N}\right) \, .
    \]
    We frequently use this parameterization of \(\Gamma\), in particular to compute the following two polynomials associated to \(S\).

\begin{definition}
Define the \emph{$h^*$-polynomial} of $S$ as
    \[
		h^*(S;z)=\sum_{i=0}^d {h_i}^* z^i\coloneqq\sum_{x\in \Gamma}z^{\age(x)}
    \]
and the \emph{local \(h^*\)-polynomial} for $S$ as
    \[
		B(S;z)=\sum_{i=1}^d b_i z^i \coloneqq \sum_{x \, \in \, \Gamma \, \cap \, {(0,1)}^{d+1}} z^{\age(x)}.
    \]
The coefficients of $h^*(S;z)$ and $B(S;z)$ form vectors called the \emph{$h^*$-vector} and \emph{local \(h^*\)-vector} of $S$, respectively.
Local \(h^*\)-polynomials of simplices are also known as \emph{box polynomials}.
\end{definition} 

\begin{example}\label{ex:hnf6box}
The simplex in Example~\ref{ex:hnf6} has $B(S;z)=z^2+z^3+z^4$ and $h^*(S;z)=1+2z^2+2z^3+z^4$.
\end{example}

It is straightforward to verify the following proposition.

\begin{proposition}\label{prop:boxpalindromic}
    For any lattice simplex \(S\), the polynomial \(B(S;z)/z\) has palindromic, i.e., symmetric, coefficients.
\end{proposition}

The $h^*$-polynomial is a main object of study in Ehrhart theory for lattice polytopes.
Local \(h^*\)-polynomials for simplices are a special case of local $h^*$-polynomials for arbitrary lattice polytopes.
For a detailed discussion of the history of local $h^*$-polynomials in this general setting, see~\cite{borgerkretschmernill}*{Section 2.4}. 
It is common that in expositions of Ehrhart theory~\cite{ccd}, these polynomials are defined using the fundamental parallelepiped of the cone over $S$, defined as follows; 
let \(\{r_0,\ldots,r_d\}\) denote the rows of the extended matrix \(A\) for \(S\). 
Then the fundamental parallelepiped is
\begin{equation}\label{eqn:fppcoeffset}
\Pi_S\coloneqq \left\{\sum_i \lambda_ir_i:0\leq \lambda_i<1\right\} \, .
\end{equation}
Multiplication by \(A\) produces a bijection between the lattice \(\Lambda(S)\) and \(\ZZ^{d+1}\), and further this takes the half-open cube \([0,1)^{d+1}\) to \(\Pi_S\).
Thus, there is a bijection between the \(\Lambda(S)\)-points in \([0,1)^{d+1}\) and the \(\ZZ^{d+1}\)-points in \(\Pi_S\).
Further, the grading of \(\ZZ^{d+1}\cap \Pi_S\) corresponding to the initial coordinate corresponds to the grading of \(\Lambda(S)\cap [0,1)^{d+1}\) by the age function.
Thus, the definitions of local-\(h^*\)- and $h^*$-polynomials given above agree with the usual definitions given in terms of the fundamental parallelepiped.

In the case of a one-row Hermite normal form simplex, certain number-theoretic conditions imply that the local \(h^*\)-polynomial and \(h^*\)-polynomial are essentially the same, as follows;

\begin{theorem}\label{thm:boxhstarconditions}
    Let \(M=\lcm(a_1,a_2,\ldots,a_{d-1},-1+\sum_ia_i)\) and let \(S\) be the simplex with non-trivial row \((a_1,\ldots,a_{d-1},N)\).
    If \(\gcd(M,N)=1\), then 
    \[
    h^*(S;z)=1+B(S;z) \, .
    \]
\end{theorem}

\begin{proof}
For a one-row Hermite normal form simplex, the parallelepiped group is generated by
\[
    \left(\frac{(-1+\sum_{i=1}^{d-1}a_i)}{N} , -\frac{a_1}{N} , -\frac{a_2}{N} , -\frac{a_3}{N} , \ldots ,-\frac{a_{d-2}}{N} ,-\frac{a_{d-1}}{N} , \frac{1}{N}\right)\, .
    \]
The condition \(h^*(S;z)=1+B(S;z)\) occurs if and only if every non-zero point in the parallelepiped group has all non-zero coordinates, since this is the criteria for all non-zero points to be in the open box \((0,1)^{d+1}\).
This holds if and only if
    \[
    \left(\frac{\ell(-1+\sum_{i=1}^{d-1}a_i)}{N} , -\frac{\ell a_1}{N} , -\frac{\ell a_2}{N} , -\frac{\ell a_3}{N} , \ldots ,-\frac{\ell a_{d-2}}{N} ,-\frac{\ell a_{d-1}}{N} , \frac{\ell}{N}\right)
    \]
    has all non-integer coordinates for every \(\ell=1,2,\ldots,N-1\).
    This is equivalent to the gcd condition in the theorem.
\end{proof}

\begin{remark}
    Note that the numerical conditions in Theorem~\ref{thm:boxhstarconditions} are identical to those identified by Hibi, Higashitani, and Li~\cite{hibihermite} as corresponding to ``shifted symmetric'' \(h^*\)-vectors.
    The symmetry that they observed is precisely the symmetry of the local \(h^*\)-polynomial from Proposition~\ref{prop:boxpalindromic}.
\end{remark}

For each polynomial \(f(z)\) with non-negative coefficients, we can associate to \(f\) a discrete probability distribution.

\begin{definition}
Given \(f(z)\in \RR_{\geq 0}[z]\), we define the \emph{distribution associated to \(f\)} to be the distribution defined by the coefficients of \(f(z)/f(1)\).
\end{definition}

We are interested in these distributions in the case of the \(h^*\)- and local \(h^*\)-polynomials for \(S\).

\begin{example}\label{ex:hnf6boxdist}
    The local \(h^*\)-and \(h^*\)-polynomials in Example~\ref{ex:hnf6box} yield the distributions \[(0,0,1/3,1/3,1/3)\] and \[(1/6,0,1/3,1/3,1/6)\, ,\] respectively.
\end{example}

\subsection{Relationship to Stapledon decompositions}
    
Stapledon established~\cite{stapledonineq} a decomposition of the $h^\ast$-polynomial of a lattice polytope into its boundary $h^\ast$-polynomial and its ``$b$-polynomial.'' 
Therefore, the $b$-polynomial of a simplex $S$ captures information about the interior of the cone over $S$, as does the local $h^\ast$-polynomial. 
In this section, we explore the relationship between these two polynomials.
Decomposition of polynomial invariants has been widely studied and, in many settings, builds on the Stapledon decomposition (e.g., \cites{BeckBraunVindas, BeckLeon, BrandenJochemko ,Jochemko, Leon}).

The $h^\ast$-polynomial of an arbitrary $d$-dimensional lattice polytope $P\in\RR^d$ is 
\[
h^\ast(P;z):=(1-z)^{d+1}\left(1+\sum_{n\geq1}|nP\cap\ZZ^d|z^n\right),
\]
and can be computed via the $h^\ast$-polynomials of the simplices in a triangulation of $P$. 
In particular, if $\mathcal{T}$ is a disjoint triangulation of $P$ into $d$-dimensional half-open simplices, then
\[
h^\ast(P;z)=\sum_{S\in \mathcal{T}}h^\ast(S;z).
\]
Here, the $h^\ast$-polynomial of a half-open simplex is computed similarly to that of a closed simplex, but with a modification to the fundamental parallelepiped. 
Suppose $S\subseteq\RR^d$ is a $d$-dimensional simplex with vertices $v_0,\hdots,v_{d}\in \ZZ^d$, where the facets opposite the first $r$ vertices are missing for some $0\leq r\leq d+1$, i.e.,
\[
S=\left\{ \lambda_0v_0+\cdots+\lambda_dv_d : \lambda_0+\cdots+\lambda_d=1 , \lambda_0,\hdots,\lambda_{r-1}>0, \lambda_r,\hdots,\lambda_d\geq0 \right\}.
\]
We define the fundamental parallelepiped of $S$ to be
\begin{equation}\label{eqn:halfopenset}
\Pi_S=\left\{ \sum_{i} \lambda_i r_i : 0<\lambda_0,\hdots,\lambda_{r-1}\leq 1, 0\leq \lambda_r,\hdots,\lambda_d<1 \right\},
\end{equation}
where $r_0,\hdots,r_d$ are the rows of the extended matrix $A$ for $S$ (that is, $r_i=(1,v_i)$).
It then holds that 
\[
h^\ast(P;z)=\sum_{S\in \mathcal{T}}h^\ast(S;z)=\sum_{S\in \mathcal{T}}\sum_{\substack{(x_0,\hdots,x_d)\\ \in\Pi_S\cap\ZZ^{d+1}}}z^{x_0}.
\] 

\begin{theorem}[Stapledon,~\cite{stapledonineq}]
If $P\in\RR^d$ is a lattice polytope and $\ell\geq1$ is the smallest integer such that $\ell P^\circ \cap\ZZ^d$ is nonempty, then there exist unique polynomials $a(P;z)$ and $b(P;z)$ such that
\[
(1+z+\cdots+z^{\ell-1})h^\ast(P;z)=a(P;z)+z^\ell b(P;z),
\]
where $a(P;z)$ and $b(P;z)$ are palindromic polynomials with nonnegative integer coefficients. 
\end{theorem}

Moreover, $a(P;z)$ is equal to the $h^\ast$-polynomial of the \textit{boundary} $\partial P$ of $P$, defined in \cite{BajoBeck} as
\[
h^\ast(\partial P; z):=\frac{h^\ast(P;z)-h^\ast(P^\circ;z)}{1-z}.
\]
The $h^\ast$-polynomial of the boundary of $P$ can also be expressed in terms of $h^\ast$-polynomials of half-open simplices. 
There is a disjoint triangulation $T$ of $\partial P$ into half-open $(d-1)$-dimensional lattice simplices (where exactly one simplex is closed), and for such a triangulation,
\[
h^\ast(\partial P;z)=\sum_{S'\in T} h^\ast(S';z).
\]

In the case where the polytope is a simplex $S$ and where $S$ contains an interior lattice point, i.e., $\ell=1$, then
\[
b(S;z)=\frac{h^\ast(S;z)-h^\ast(\partial S;z)}{z}
\]
captures information about the interior of the cone over $S$. 
The local $h^\ast$-polynomial $B(S;z)$ also captures information about the interior of this cone, and in this case, we are able to directly compare these two polynomials.

\begin{proposition}
Let $S$ be a lattice simplex with an interior lattice point, and let 
\[
b(S;z)=\frac{h^\ast(S;z)-h^\ast(\partial S;z)}{z}
\]
be its $b$-polynomial as in Stapledon's decomposition. 
Then $zb(z)$ is coefficient-wise bounded above by the local \(h^*\)-polynomial $B(S;z)$ of $S$.
\end{proposition}
\begin{proof}
Use the facets of \(S\) as the maximal faces in a triangulation \(T\) of $\partial S$.
Arrange the facets of \(S\) into disjoint half-open simplices such that exactly one simplex is closed, e.g., using a visibility construction as in \cite{CRT}*{Chapter 5} or \cite{BajoBeck}. Let
\[
\Pi_T=\bigcup_{S'\in T}\Pi_{S'}
\]
be the union over the fundamental parallelepipeds of the half-open simplices in $T$. 
Then using~\eqref{eqn:halfopenset} we obtain
\begin{align*}
		zb(S;z)&=h^\ast(S;z)-h^\ast(\partial S,z)=\left(\sum_{z\in \Pi_S\cap\ZZ^{d+1}} z^{x_0}\right)-\left(\sum_{x\in \Pi_T\cap\ZZ^{d+1}}z^{x_0}\right)\\
		&= \left(\sum_{z\in (\Pi_S\setminus \Pi_T)\cap\ZZ^{d+1}}z^{x_0}\right)  -\left(\sum_{z\in (\Pi_T\setminus \Pi_S)\cap\ZZ^{d+1}}z^{x_0}\right).
  \end{align*}
Note that \(\Pi_T\) depends on the specific choice of half-open simplices in \(T\), but the polynomial 
\[
\sum_{x\in \Pi_T\cap\ZZ^{d+1}}z^{x_0}
\]
is independent of this choice. 
Observe that $\Pi_T$ contains all points in $\Pi_S$ where some coefficient \(\lambda_i\) as in~\eqref{eqn:fppcoeffset} is \(0\) (namely, all coefficients of the row vectors of $A$ corresponding to vertices that are opposite a boundary facet containing the point) and $\Pi_T$ contains no points in $\Pi_S$ where each coefficient is positive. 
Therefore, the lattice points in $\Pi_S\setminus \Pi_T$ are precisely those with no coefficient equal to 0, that is, the points counted by the local \(h^*\)-polynomial $B(S;z)$.
Thus,
	\[
		zb(S;z)=B(S;z)-\left(\sum_{z\in (\Pi_T\setminus \Pi_S)\cap\ZZ^{d+1}}z^{x_0}\right).
	\]
	In particular, $zb(S;z)$ is upper-bounded by $B(S;z)$, coefficient-wise.
\end{proof}

We remark that the same result does not follow if $S$ does not contain an interior point, as now there is a factor in front of $h^\ast(S;z)$ in Stapledon's decomposition:
\[
	z^\ell b(S;z)=(1+z+\hdots+z^{\ell-1})h^\ast(S;z)-h^\ast(\partial S;z).
\]
For example, $S=\text{conv}\{(0,0,0),(1,0,0),(0,1,0),(1,1,4)\}$ has local \(h^*\)-polynomial
\[
	B(S;z)=4z^2
\]
and $b$-polynomial
\[
	b(S;z)=2+2z \text{ (and $z^\ell b(S;z)=2z^2+2z^3$)}.
\]

We conclude this subsection with an observation about the implication of Stapledon's decomposition for \textit{shifted symmetric polytopes}, as explored by Hibi, Higashitani, and Li in~\cite{hibihermite}.

\begin{definition}
	The $h^\ast$-polynomial $h^\ast_P(z)=h^\ast_0+h^\ast_1 z+\hdots h^\ast_d z^d$ of a lattice polytope $P$ is \emph{shifted symmetric} if $h^\ast_i=h^\ast_{d+1-i}$ for $i=1,\hdots,d$.
\end{definition}

Hibi, Higashitani, and Li~\cite{hibihermite} proved that for a one-row Hermite normal form simplex \(S\) satisfying the conditions of Theorem~\ref{thm:boxhstarconditions}, the \(h^*\)-polynomial of \(S\) is shifted symmetric.
The following proposition due to Higashitani uses Stapledon decompositions to show that any lattice polytope with a shifted symmetric \(h^*\)-polynomial is a simplex with a similar geometric property to those in Theorem~\ref{thm:boxhstarconditions}.

\begin{proposition}[Higashitani~\cite{higashitanishifted}]
	If \(P\) is a lattice polytope such that $h^\ast_P(z)$ is shifted symmetric, then $P$ is a simplex with unimodular facets.
 Thus, \(h^*(P;z)=1+B(P;z)\).
\end{proposition}


\section{Two Illustrative Examples}\label{sec:examples}

In this section, we analyze two examples of one-row Hermite normal form simplices that illustrate important phenomena and proof techniques.
We focus on those simplices where the non-trivial row is either the ``all-ones'' vector $\mleft(1,\dots,1,N\mright)$ or the ``geometric sequence'' vector $\mleft(q^{k-1}, \dots, q,1,q^k\mright)$.
For the all-ones simplex, in Corollary~\ref{cor:allonescharacteriation}  we classify the values of \(N\) for which the local \(h^*\)-polynomial is unimodal. 
For the geometric series simplex, we prove in Theorem~\ref{thm:geometricunimodalbox} that the local \(h^*\)-polynomial is always unimodal.
These results are independently interesting and they motivate the results in Section~\ref{sec:limits}.

\subsection{Non-trivial row \texorpdfstring{$\mleft(1,\dots,1,N\mright)$}-}\label{sec:allones}

In this subsection, let $N \in \ZZ$ be a fixed integer with\\ $N>1$.
We consider a $d$-simplex $S$ in one-row Hermite normal form with last row $\mleft(1, \dots, 1, N\mright)$.
By~\eqref{eq:onerowinverse} we have that
\[
	\begin{bmatrix}
		1 & 0 & 0 & 0 & \cdots & 0 & 0 & 0\\
		1 & 1 & 0 & 0 & \cdots & 0 & 0 & 0\\
		1 & 0 & 1 & 0 & \cdots & 0 & 0 & 0\\
		1 & 0 & 0 & 1 & \cdots & 0 & 0 & 0\\
		\vdots & \vdots & \vdots & \vdots & \ddots & \vdots & \vdots & \vdots \\		1 & 0 & 0 & 0 & \cdots & 1 & 0 & 0\\
		1 & 0 & 0 & 0 & \cdots & 0 & 1 & 0\\
		1 & 1 & 1 & 1 & \cdots & 1 & 1 & N
	\end{bmatrix}^{-1} =
	\begin{bmatrix}
		\hphantom{-}1 & \hphantom{-}0 & \hphantom{-}0 & \hphantom{-}0 & \cdots & \hphantom{-}0 & \hphantom{-}0 & \hphantom{-}0\\
		-1 & \hphantom{-}1 & \hphantom{-}0 & \hphantom{-}0 & \cdots & \hphantom{-}0 & \hphantom{-}0 & \hphantom{-}0\\
		-1 & \hphantom{-}0 & \hphantom{-}1 & \hphantom{-}0 & \cdots & \hphantom{-}0 & \hphantom{-}0 & \hphantom{-}0\\
		-1 & \hphantom{-}0 & \hphantom{-}0 & \hphantom{-}1 & \cdots & \hphantom{-}0 & \hphantom{-}0 & \hphantom{-}0\\
		\vdots & \vdots & \vdots & \vdots & \ddots & \vdots & \vdots & \vdots \\		-1 & \hphantom{-}0 & \hphantom{-}0 & \hphantom{-}0 & \cdots & \hphantom{-}1 & \hphantom{-}0 & \hphantom{-}0\\
		-1 & \hphantom{-}0 & \hphantom{-}0 & \hphantom{-}0 & \cdots & \hphantom{-}0 & \hphantom{-}1 & \hphantom{-}0\\
		\frac{(d-2)}N & -\frac{1}N & -\frac{1}N & -\frac{1}N & \cdots & -\frac{1}N & -\frac{1}N & \frac1N
	\end{bmatrix} \, .
\]
Hence, the parallelepiped group of $S$ is given by
\[
	\Gamma = \mleft(\ZZ^{d+1}+\ZZ{\mleft(\frac{(d-2)}N, -\frac1N, \dots, -\frac1N,\frac1N\mright)}^t\mright)\bigg/\; \ZZ^{d+1}
\]
Recall that the \emph{$h^*$-vector} $h^*=(h_0^*,h_1^*, \dots, h_d^*)$ of $S$ is given by
\[
	h_i^* = \#\mleft\{ x \in \Gamma : \age(x) = i \mright\} \, .
\]
We can compute the local \(h^*\)-polynomial $B(S;z) = \sum_{i=0}^d b_i z^i$ as follows; parameterizing the points in \(\Gamma\) using \(0\leq k\leq N-1\),
\begin{align*}
	b_i &= \left|\mleft\{ x = \mleft(\frac{-k(d-2)}N,\frac kN, \dots, \frac kN, -\frac kN\mright) + \ZZ^{d+1}\in \Gamma : \age(x)=i, \frac{-k(d-2)}N \notin \ZZ \mright\}\right|\\
	&=\left|\mleft\{ k=1, \dots, N-1 : i=1+\fractional{\frac{-k(d-2)}N}+\frac{k(d-2)}N, \fractional{\frac{-k(d-2)}N}\neq0\mright\}\right|\\
	&=\left|\mleft\{ k=1, \dots, N-1 : i=1 - \floor{\frac{-k(d-2)}N}, N \nmid k(d-2) \mright\}\right|\\
	&=\left|\mleft\{ k=1, \dots, N-1 : i=1 + \ceil{\frac{k(d-2)}N}, N \nmid k(d-2) \mright\}\right|.
\end{align*}

Our goal in this subsection is to prove the following theorem.

\begin{theorem}\label{thm:allonesboxpoly}
    For the $d$-simplex $S$ in one-row Hermite normal form with last row \\$\mleft(1, \dots, 1, N\mright)$, the local \(h^*\)-polynomial of \(S\) has every non-zero coefficient in \(\{q,q+1\}\) where \(N = (d-2)q+r\) for some \(0\leq r \leq d-3\).
    Further, the local \(h^*\)-polynomial is unimodal if and only if \(r\in\{0,1,2,d-3\}\) and it has constant coefficients if and only if \(r\in\{0,1\}\).
\end{theorem}

\begin{example}\label{ex:allonesbox}
Figure~\ref{fig:allonesdist} displays the local \(h^*\)-polynomial distribution for the all ones non-trivial row with \(d=17\) and \(N=22\cdot 15+1\).
As claimed by Theorem~\ref{thm:allonesboxpoly}, since \(r=1\), the local \(h^*\)-polynomial is both constant and unimodal.
\end{example}

The remainder of this subsection is devoted to a proof of Theorem~\ref{thm:allonesboxpoly}.
To simplify the notation, let $a, N \in \NN$ be two natural numbers.
Note that in our above setup, we have $a=d-2$.
Throughout the remainder of this section, we divide $N$ by $a$ with a remainder, i.e., $N=a\cdot q + r$ for some $q, r \in \NN$ with $0\le r < a$.
We define a vector $\alpha = (\alpha_1, \dots, \alpha_a) \in \ZZ^a$ as follows;
\[
	\alpha_i = \#\mleft\{k = 1, \dots, N-1 : \ceil{\frac{ak}N} = i, N \nmid ak \mright\}.
\]
Thus, when \(a=d-2\), we have that the local \(h^*\)-polynomial coefficient is 
\[
b_i=\alpha_{i-1}
\]
and hence, the coefficients of the local \(h^*\)-polynomial are a shift of the vector \(\alpha\).
We focus on the case that $a$ and $N$ are coprime and hence assume until the end of the section that $\gcd(a,N)=1$.
This assumption means we do not need to consider the condition $N \nmid ak$.

Let us begin with an example to illustrate our general strategy.

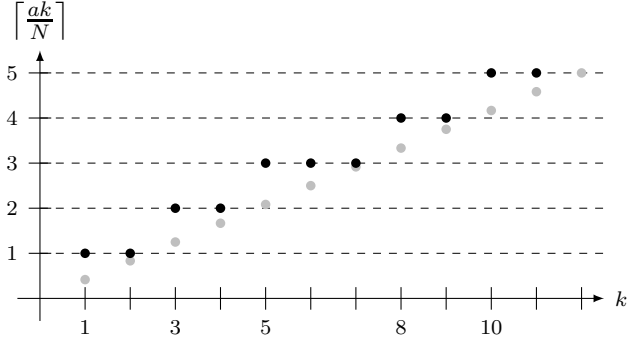
\begin{figure}[ht]
		\begin{tikzpicture}[scale=.6]
			\draw[-latex] (-.5,0) -- (12.5,0) node[right] {\tiny$k$};
			\draw[-latex] (0,-.5) -- (0,5.5) node[above] {$\ceil{\frac{ak}N}$};
			\foreach \l in {1,...,5}
				{
					\draw[very thin,dashed] (0,\l) -- (12.5,\l);
					\draw (-.25,\l) node[left] {\tiny$\l$} -- (.25,\l);
				}
			\foreach \k in {1,...,12}
				{
					\draw (\k,-.25) -- (\k,.25);
					\fill[fill=lightgray] (\k, 5*\k/12) circle (3pt);
				}
			\fill (1,1) circle (3pt);
			\fill (2,1) circle (3pt);
			\fill (3,2) circle (3pt);
			\fill (4,2) circle (3pt);
			\fill (5,3) circle (3pt);
			\fill (6,3) circle (3pt);
			\fill (7,3) circle (3pt);
			\fill (8,4) circle (3pt);
			\fill (9,4) circle (3pt);
			\fill (10,5) circle (3pt);
			\fill (11,5) circle (3pt);

			\node[below] at (1,-.25) {\tiny$1$};
			\node[below] at (3,-.25) {\tiny$3$};
			\node[below] at (5,-.25) {\tiny$5$};
			\node[below] at (8,-.25) {\tiny$8$};
			\node[below] at (10,-.25) {\tiny$10$};
		\end{tikzpicture}
		\caption{The distribution of the values $\ceil{ak/N}$ for $k=1,\dots,N$.\label{ex:dist-ceil-vals}}
	\end{figure}

\begin{example}
Suppose that $a=5$ and $N=12$.
Thus, we have $N = q \cdot a + r$ where $q=2$ and $r=2$.
From the definition, we see that  the $\alpha_i$'s depend on the distribution of the values $\ceil{ak/N}$ for $k = 1, \dots, N-1$.
Figure~\ref{ex:dist-ceil-vals} illustrates this distribution with respect to the multiples $k$.
From Figure~\ref{ex:dist-ceil-vals}, we deduce $\alpha=(2,2,3,2,2)$.
Notice that the fraction $12\cdot a/N$ does not contribute to the $\alpha$-vector since $N=12$ and thus $N\mid 12\cdot a$.
Furthermore, observe that fractions $ak/N$ contributing to the same entry $\alpha_i$ (for some fixed $i=1,\dots,a$) appear consecutively and the possible values for $\alpha_i$ are either $q=2$ or $q+1=3$.

Our proof strategy relies on studying the values $i$ for which $\alpha_i=q$ (respectively $\alpha_i=q+1$).
Fix $i=0,\dots,a-1$ and consider the smallest $k=0,\dots,N-1$ such that $\floor{ak/N}=i$.
Notice that for our example with $a=5$ and $N=12$, we have that 
	\[
	ak = 5k \equiv \overline{-i\cdot 2} \pmod{12} \equiv \overline{-ir} \pmod{N},
	\]
	where $\overline{\parbox{1em}{\centering$\cdot$}}$ denotes smallest residue mod $a$.
	Furthermore, observe that for our example we have
	\[
		\alpha_{i+1} = \begin{cases}
			q+1 & \text{if $1\le a - \overline{ir} \le r-1$}\\
			q & \text{if $r\le a - \overline{ir} \le a$}.
		\end{cases}
	\]
This is true in general and is proven in Proposition~\ref{prop:delta-i-vals}.

Since $a$ and $N$ are coprime, it follows that $\gcd(a,r) = 1$, and thus $\ZZ/a\ZZ = \{ \overline{ir} : i=0, \dots a-1 \}$.
As $r-1=2-1=1$, it follows that exactly one entry of the vector $\alpha$ is equal to $q+1$, and thus $\alpha$ is unimodal.
With similar arguments, one can identify exactly when the vector $\alpha$ is unimodal in general.
This is shown in Theorem~\ref{thm:delta-unimodal-a-prime}.
\end{example}

\begin{proposition}\label{prop:delta-i-vals}
Let $i=0,\dots,a-1$ and let $k=0,\dots,N-1$ be the smallest multiple with $\floor{ak/N}=i$.
Then $ak \equiv \overline{-ir} \pmod N$ and
	\[
		\alpha_{i+1} = \begin{cases}
			q+1 & \text{if $1\le a - \overline{ir} \le r-1$}\\
			q & \text{if $r \le a - \overline{ir} \le a$}
		\end{cases}\;,
	\]
	where $\overline{\parbox{1em}{\centering$\cdot$}}$ denotes smallest residue mod $a$.
\end{proposition}
\begin{proof}
The equation $ak \equiv \overline{-ir} \pmod N$ can be verified for $i=0$.
Suppose $i=1,\dots, a-1$.
Then the smallest $k = 0, \dots, N-1$ with $\floor{ak/N} = i$ satisfies \[a(k-1) < iN \le ak\] and thus \begin{equation}\label{eqn:condition} 
0<iN-a(k-1)<a.\end{equation}
Observe that \[iN = iaq + ir = a(k-1) + x\] for some $x\in\ZZ$ with $x \equiv ir \pmod a$.
With the condition from (\ref{eqn:condition}), we deduce that $x = \overline{ir}$.
Thus, \[iN + (a-\overline{ir})= a(k-1) + \overline{ir} + (a-\overline{ir})= ak,\] which yields $ak \equiv \overline{-ir} \pmod N$.

Notice that for $i=1,\dots, a-1$, we have that
	\[
		\alpha_{i+1}= \# \mleft\{ l \in \ZZ : i + \frac{a-\overline{ir}}N \le i + \frac{al + (a - \overline{ir})}N < i+1 \mright\}.
	\]
Furthermore, for $i=0$ it follows that $\alpha_1 = \{ l \in \ZZ : a/N \le (al + a)/N < 1\}$.
Hence, the second statement of our proposition is equivalent to
	\[
		\#\mleft\{l\in\ZZ : i + \frac{a-\overline{ir}}N \le i+\frac{al+(a-\overline{ir})}N < i+1\mright\} = \begin{cases}
			q+1 & \text{if} \quad 1 \le a - \overline{ir} \le r-1\\
			q   & \text{if} \quad r \le a - \overline{ir} \le a
		\end{cases}\;.
	\]
We begin with the case $1\le a - \overline{ir} \le r-1$.
Clearly, we have that
	\[
		i+\frac{a-\overline{ir}}N \le \underbrace{i+\frac{a-\overline{ir}}N, i+\frac{a+(a-\overline{ir})}N, \dots, i+\frac{a \cdot q + (a - \overline{ir})}N}_{q+1 \; \text{many}} < i+1.
	\]
Furthermore, notice that $i+(a\cdot (q+1) + (a-\overline{ir})) / N > i+1$.
Hence, the first case follows.

Next, suppose that $r\le a - \overline{ir} \le a-1$.
We have that
	\[
		i+\frac{a-\overline{ir}}N \le \underbrace{i+\frac{a-\overline{ir}}N, i+\frac{a+(a-\overline{ir})}N, \dots, i+\frac{a \cdot (q-1) + (a-\overline{ir})}N}_{q \; \text{many}}  < i+1.
	\]
Furthermore, notice that $i+(a \cdot q + (a - \overline{ir}))/N>i+1$.
Hence, the second case follows.
\end{proof}

\begin{corollary}\label{cor:oscillate}
For $a$ and $N$ such that $\gcd(a,N)=1$, we have  $\alpha_i \in \{q,q+1\}$.
\end{corollary}

The coefficients of the local \(h^*\)-polynomial are a simple shift of the $\alpha$-vector.  Corollary~\ref{cor:oscillate} implies that the local \(h^*\)-polynomial coefficients are unimodal precisely when the values of $q+1$ occur in a consecutive sequence in the vector.
The following theorem characterizes when this occurs.

\begin{theorem}\label{thm:delta-unimodal-a-prime}
Given $a$ and $N$ with $N=aq+r$, the $\alpha$-vector is unimodal if and only if $r \in \{0, 1, 2, a-1\}$.
Furthermore, the $\alpha$-vector is constant if and only if $r\in\{0,1\}$.
\end{theorem}

\begin{proof}
Since $N$ and $a$ are coprime, so are $a$ and $r$, thus $\{\overline{ir}|i \in [a-1]\} = [a-1]$.
Thus, the $\alpha$-vector has exactly $r-1$ entries equal to $q+1$.
It is enough to show that these $r-1$ entries equal to $q+1$ cannot all be consecutive.
Equivalently, that the values of $a- \overline{ir}$ between $1$ and $r-1$ cannot be consecutive in the sequence \begin{equation}\label{eq:ir}
a, a- \overline{r}, a- \overline{2r}, \dots, a- \overline{(a-1)r}.
\end{equation}
We note that an element of this sequence is obtained from the previous one by subtracting $r$ if the result of the subtraction is positive and by adding $a-r$ otherwise.

To prove the ``if'' part of the theorem, we note that the $\alpha$-vector in the case $r=0$ is trivial.
In the case $r=1$ by the above formula all entries are equal to $q$.
In the case $r=2$ it has exactly one $q+1$ entry, not the first, and thus in all these cases the $\alpha$-vector is unimodal.
For $r=a-1$, the sequence~\eqref{eq:ir} is simply $a, 1, 2, \dots, a-1$, so there are exactly two $q$ entries in the $\alpha$-vector, the first and the last, making the vector unimodal.

Now for the ``only if'' direction: let $3 \leq r \leq a-2$.
There are $r-1$ entries equal to $q+1$, so at least two and at most $a-3$ (i.e., at least three $q$ entries).
It is enough to show that these $r-1$ entries equal to $q+1$ cannot all be consecutive.
That is, in the sequence~\eqref{eq:ir} the values between $1$ and $r-1$ cannot all be consecutive.
We distinguish two cases: 

	\begin{itemize}
		\item If $r > a-r$ (i.e. $r> \frac{a}2$); suppose $a=k(a-r) +t$, with $0 \le t < a-r$.
			Then the sequence (\ref{eq:ir}) begins with $a, a-r, \dots, k(a-r)$, where the second entry already corresponds to a $q+1$ entry in the $\alpha$-vector, since $a-r < r$.
			Since $r>\frac a2$ and $a-r \ge 2$, it follows that $r \cdot (a-r) >a$, and thus $k \le r-1$.
			Furthermore, notice that $k \cdot (a-r) = a - t > r$.
			We conclude that $\alpha_2=q+1$, $\alpha_k=q$, and $k \le r-1$.
			Hence, there exists $i>k$ with $\alpha_i=q+1$, which means that the sequence is not unimodal.
			
		
		\bigskip
		\item If $r < a-r$ (i.e., $r<\frac a2$), we proceed similarly.
			Let us write $a= \ell \cdot r + s$ with $0 \le s <r$.
			Then the sequence~\eqref{eq:ir} begins with $a, a-r, \dots, a-\ell r$, where the last entry satisfies $a-\ell r=s<r$, i.e., $\alpha_\ell=q+1$.
			In other words, the sequence~\eqref{eq:ir} starts with consecutive $q$'s followed by one or more $(q+1)$'s.
			Notice that $(\ell+1)r > \ell r +s=a$, so that $\overline{(\ell+1)r}=(\ell+1)r-a = r - s$.
			Hence, $a-\overline{(\ell+1)r}=a-r+s$.
			Since $r<\frac a2$, it follows that $a-\overline{(\ell+1)r}=a-r+s>\frac a2>r$, i.e., $\alpha_{\ell+1}=q$.
			Together with the symmetry $\alpha_i=\alpha_{a+1-i}$ for all $i=1,\dots,a$, we conclude this sequence is not unimodal.
	\end{itemize}
\end{proof}

Using our result for the case where $\gcd(a,N)=1$, we can extend our result to the general setting, which implies Theorem~\ref{thm:allonesboxpoly}.

\begin{corollary}\label{cor:allonescharacteriation}
    Let $a, N \in \NN$ and let $b = \gcd(a,N) \ne 1$.
    Divide $N$ by $a$ with remainder, i.e., write $N = aq + r$ for some $q,r \in \NN$ with $0 \le r <a$.
    Then the $\alpha$-vector is unimodal if and only if $r = 0$ or $r = b$.
\end{corollary}

\begin{proof}
Let $a = a' \cdot b$ and $N = N' \cdot b$ for some $a', N' \in \NN$.
Firstly, observe that $\alpha$ is a $b$-fold concatenation of the vector $\alpha'$ that corresponds to the numbers $a'$ and $N'$.
Therefore, $\alpha$ is unimodal if and only if $\alpha'$ is constant, and hence it remains to check when this is the case.
    
Let us divide $N'$ by $a'$ with remainder, i.e., $N' = a' \cdot q' + r'$ for some $q', r' \in \NN$ with $0 \le r' <a'$.
By Theorem~\ref{thm:delta-unimodal-a-prime}, $\alpha'$ is constant if and only if $r' \in \{0,1\}$.
With $N = b \cdot N' = q' \cdot (a' b) + r' \cdot b = q' \cdot a + (r' \cdot b)$, the statement follows.
\end{proof}

\subsection{Non-trivial row \texorpdfstring{$\mleft(q^{k-1}, \dots, q,1,q^k\mright)$}- }\label{sec:geometric}

In this subsection, let $k \in \ZZ$ be a fixed positive integer and $q \in \ZZ_{\geq 2}$. 
Consider the one-row Hermite normal form simplex \(S\) with final row $\mleft(q^{k-1}, \dots, q,1,q^k\mright)$.
We study this as a test case for one-row Hermite normal form simplices where the final row has ``well-spaced'' entries.

To begin, in the following proposition we show that any \(S\) of this form does not have the Integer Decomposition Property (IDP).
Recall that a simplex \(S\) has the IDP if every lattice point in the non-negative cone generated by the colums of \(A\) is a sum of lattice points in the cone having last coordinate equal to \(1\).
It is known that polytopes with a unimodular triangulation have the IDP, and thus these simplices do not have a regular unimodular triangulation.
This is noteworthy because, as we will see in Theorem~\ref{thm:geometricunimodalbox} below, the local \(h^*\)-polynomial of every such \(S\) is unimodal, as illustrated in Figure~\ref{fig:q3_k12_geometricdist}.
This family of simplices demonstrates that the existence of regular unimodular triangulations is not the only source of unimodality for local \(h^*\)-polynomials of one-row Hermite normal form simplices.

\begin{proposition}
The simplex \(S\) does not have the IDP.
\end{proposition}

\begin{proof}
We begin by observing that
	\[
		1< \frac1q + \dots + \frac1{q^k} + \mleft( 1 - \frac1{q^k} \mright) <2.
	\]
Let \(\lambda \in (0,1)\) such that \(\lambda + \sum_{i=1}^k \frac1{q^i} + \mleft( 1 - \frac1{q^k} \mright) = 2\), so that the convex combination
	\[
		b = (q^{k-1}, \dots, q, 1, q^k - 1) = \lambda v_0 + \sum_{i=1}^k \frac1{q^i} v_i + \mleft( 1 - \frac1{q^k} \mright) v_{k+1} \in 2S.
	\]
Assume towards a contradiction that \(b = b' + b''\) for two elements \(b', b'' \in S \cap \ZZ^{k+1}\).
	Let us write \(b' = (b'_0, \dots, b'_{k+1})\) and \(b'' = (b''_0, \dots, b''_{k+1})\) for \(b'_i, b''_j \in \ZZ_{\ge0}\).
Then the \(k\)-th coordinate of \(b'\) or \(b''\) has to vanish, say \(b'_k = 0\).
We can write \(b'=\sum_{i=0}^{k+1} \mu'_i v_i\), with \(\mu'_i \in [0,1]\) such that \(\sum_{i=0}^{k+1} \mu'_i=1\).
Since \(0 = \mu'_{k+1} + \mu'_k\), it follows that \(\mu'_{k+1} = \mu'_k = 0\).
Thus, \(b''_{k+1} = q^k - 1\).

Let us write \(b'' = \sum_{i=0}^{k+1} \mu''_i v_i\) for \(\mu''_i \in [0,1]\) with \(\sum_{i=0}^{k+1} \mu''_i = 1\).
Since \(b\) is in the interior of \(2S\) and \(b'\) is on the boundary of \(S\), it follows that \(b''\) is in the interior of \(S\), or in other words \(\mu''_i \in (0,1)\).
From \(b''_{k+1} = q^k - 1\) we conclude that \(\mu'_{k+1} = \frac{q^k - 1}{q^k}\).
Observe that the entry $b''_{k+1}$ completely determines the vector $b''$.
Indeed, for \(i = 1, \dots, k\), since \(b''_i = \mu''_{k+1} q^{k-i} + \mu''_i\) is an integer and $\mu''_i \in (0,1)$, we have that \(\mu''_i = \frac1{q^i}\), which also fixes \(b''_i\).
Note that \(\sum_{i=1}^{k+1} \mu''_i > 1\) - a contradiction.
Hence, the element \(b \in 2S \cap \ZZ^{k+1}\) is not a sum of two elements in \(S \cap \ZZ^{k+1}\).
\end{proof}

\begin{theorem}\label{thm:geometricunimodalbox}
    For any integers \(q\geq 2\) and \(k\geq 2\), the simplex \(S\) with non-trivial row \[\mleft(1, q^{k-1}, \dots, q,1,q^k\mright)\] has a unimodal local \(h^*\)-polynomial.
\end{theorem}

The remainder of this subsection is devoted to a proof of Theorem~\ref{thm:geometricunimodalbox}.
We begin by computing the parallelepiped group.
Recall that this is done by lifting the vertices \(v_0, v_1, \dots, v_{k+1}\) to height one and using these vectors as the rows of a matrix \(A\).
Let us denote the rows of the inverse matrix \(A^{-1}\) by \(r_1(A), \dots, r_{k+2}(A)\).
Then the parallelepiped group is \(\Gamma = \ZZ^{k+2} + \sum_{i=1}^{k+2} \ZZ r_i(A)\).
In our case, the inverse matrix \(A^{-1}\) is given by:
\[
	A^{-1} = \begin{pmatrix}
		\hphantom{-}1 & 0 & 0 & 0 & \cdots & 0 & 0 & 0\\
		-1 & 1 & 0 & 0 & \cdots & 0 & 0 & 0\\
		-1 & 0 & 1 & 0 & \cdots & 0 & 0 & 0\\
		-1 & 0 & 0 & 1 & \cdots & 0 & 0 & 0\\
		\multicolumn{8}{c}{\cdotfill}\\
		-1 & 0 & 0 & 0 & \cdots & 1 & 0 & 0\\
		-1 & 0 & 0 & 0 & \cdots & 0 & 1 & 0\\
		\frac{q^{k-2}+q^{k-3} + \dots + q + 1}{q^{k-1}} & -\frac1{q} & -\frac1{q^2} & -\frac1{q^3} & \cdots & -\frac1{q^{k-1}} & -\frac1{q^k} & \frac1{q^k}
	\end{pmatrix}.
\]
Hence, the parallelepiped group of \(S\) is given by
\[
	\Gamma = \mleft(\ZZ^{k+2} + \ZZ \mleft((-1) \cdot \frac{q^{k-2}+q^{k-3}+\dots+q+1}{q^{k-1}}, \frac1q, \frac1{q^2}, \dots, \frac1{q^k}, (-1) \cdot \frac1{q^k}\mright)\mright)/\ZZ^{k+2}.
\]
Thus, the coefficients of the local \(h^*\)-polynomial \(B(S;z) = \sum_{i=0}^{k+2}b_iz^i\) satisfy
\[
	b_i = \mleft|\mleft\{x=(x_1,\dots,x_{k+2}) + \ZZ^{k+2} \in \Gamma : \age(x) = i, x_i \not \in \ZZ \; \text{for all}\; i\mright\}\mright|.
\]
It follows that \(\sum_{i=0}^{d}b_i = (q-1) \cdot q^{k-1}\).
To prove unimodality of the local \(h^*\)-polynomial coefficients, it suffices to show that the sequence \((\delta'_1, \delta'_2, \dots, \delta'_{k-1})\) is a unimodal sequence where
\begin{align*}
	\delta'_i &= \mleft|\mleft\{\ell=1, \dots, q^k \colon q \nmid \ell, \sum_{i=1}^{k-1}\fractional{\frac\ell{q^i}} + \fractional{-\sum_{i=1}^{k-1}\frac\ell{q^i}}=i\mright\}\mright|\\
	&= \mleft|\mleft\{ \ell = 1, \dots, q^k, \colon q\nmid \ell, \ceil{\sum_{i=1}^{k-1}\fractional{\frac\ell{q^i}}}=i\mright\}\mright|.
\end{align*}
We can further simplify this by defining the sequence \((\delta_1,\dots,\delta_{k-1})\) where
\[
	\delta_i = \mleft|\mleft\{\ell = 1, \dots, q^{k-1} \colon q \nmid \ell, \smash{\underbrace{\ceil{\sum_{i=1}^{k-1} \fractional{\frac\ell{q^i}}}}_{\eqqcolon\age(\ell)}} = i \vphantom{\sum_{i=1}^{k-1} \fractional{\frac\ell{q^i}}}\mright\}\mright|\vphantom{\underbrace{\ceil{\sum_{i=1}^{k-1} \fractional{\frac\ell{q^i}}}}_{\eqqcolon\age(\ell)}}
\]
and observing that \(\delta'_i = q \cdot \delta_i\), which is obtained by writing \(\ell\) in base \(q\).
Thus, unimodality of the local \(h^*\)-polynomial is equivalent to unimodality of the sequence \(\delta_i\).

Let us next express \(\ell = 1, \dots, q^{k-1}\) in base \(q\), i.e., we write \(\ell = \sum_{i=0}^{k-2}c_i q^i\) for natural numbers \(c_i \in\{ 0, 1, \dots, q-1\}\).
Then \(\age(\ell)\) can be rewritten as follows;
\begin{align*}
	\age(\ell) &= \ceil{\sum_{i=1}^{k-1} \fractional{\frac\ell{q^i}}} = \ceil{\sum_{i=1}^{k-1}\fractional{\frac{\sum_{j=0}^{k-2}c_j q^j}{q^i}}} = \ceil{\sum_{i=1}^{k-1}\fractional{\sum_{j=0}^{i-1}c_j q^{j-i}}}\\
	&= \ceil{\sum_{i=1}^{k-1}\sum_{j=0}^{i-1}c_j q^{j-i}}=\ceil{\sum_{j=0}^{k-2}c_j\sum_{i=1}^{k-1-j}q^{-i}} = \ceil{\sum_{j=0}^{k-2}c_j \frac{1-q^{-k+1+j}}{q-1}}\\
	&= \ceil{\frac{c_0 + c_1 + \dots + c_{k-2}}{q-1} - \frac{\ell}{q^{k-1}(q-1)}} \text{.}
\end{align*}

The following proposition considerably simplifies the computation of \(\age(\ell)\).
\begin{proposition}
	For all \(\ell = 1, \dots, q^{k-1}\), we have that
	\[
		\age(\ell) = \ceil{\frac{c_0 + c_1 + \dots + c_{k-2}}{q-1}},
	\]
	where \(\ell = c_0 + c_1 q + c_2q^2 + \dots + c_{k-2}q^{k-2}\) is the expression of \(\ell\) to base \(q\).
\end{proposition}

\begin{proof}
We begin by observing that (recall that \(q\ge2\))

	\[
		\frac{\ell}{q^{k-1}(q-1)} < \frac1{q-1} \le 1.
	\]
	Let us write the fraction \((c_0 + \dots + c_{k-2})/(q-1) = r + \delta/(q-1)\) for some natural numbers \(r, \delta \in \ZZ\) where \(\delta = 0, 1, \dots, q-2\).
We distinguish two cases.

Suppose \(\delta = 0\).
Since \(\ell/(q^{k-1}(q-1)) < 1\), it follows that
	\[
		\ceil{\frac{c_0 + \dots + c_{k-2}}{q-1}} = \ceil{r} = \ceil{r - \smash{\underbrace{\frac{\ell}{q^{k-1}(q-1)}}_{<1}}\vphantom{\frac{\ell}{q^{k-1}(q-1)}}} = \age(\ell). \vphantom{\underbrace{\frac{\ell}{q^{k-1}(q-1)}}_{<1}} \,.
	\]

Next suppose that \(\delta > 0\).
Since \(\ell/(q^{k-1}(q-1))<1/(q-1)\), it follows that
	\[
		\ceil{\frac{c_0 + \dots + c_{k-2}}{q-1}} = \ceil{r + \frac{\delta}{q-1}} = \ceil{r + \frac{\delta}{q-1} - \smash{\underbrace{\frac{\ell}{q^{k-1}(q-1)}}_{<1/(q-1)}}\vphantom{\frac{\ell}{q^{k-1}(q-1)}}} = \age(\ell). \vphantom{\underbrace{\frac{\ell}{q^{k-1}(q-1)}}_{<1/(q-1)}}
	\]
\end{proof}

With the previous proposition, it follows that
\[
	\delta_i = \mleft| \bigsqcup_{j=(i-1)(q-1)+1}^{i(q-1)}\mleft\{ (c_0, \dots, c_{k-2}) \in [1,q-1] \times {[0,q-1]}^{k-2} \cap \ZZ^{k-1} \colon \sum_{m=0}^{k-2} c_m = j \mright\} \mright|.
\]
The number of tuples \((c_0, \dots, c_{k-2}) \in [1, q-1] \times {[0, q-1]}^{k-2} \cap \ZZ^{k-1}\) which sum up to the integer \(j\) coincides with the \(j\)-th coefficient of the following polynomial
\[
	f = (t + t^2 + \dots + t^{q-1}) \cdot {(1 + t + t^2 + \dots + t^{q-1})}^{k-2} = \sum_{m=0}^{(k-1)(q-1)} \alpha_m t^m \in \ZZ[t].
\]
We get \(\delta_i = \sum_{j=(i-1)(q-1)+1}^{i(q-1)} \alpha_j\).
Since the sequence \((\delta_1, \dots, \delta_{k-1})\) is palindromic, it suffices to show that \(\delta_1 \le \delta_2 \le \dots \le \delta_{\floor{(k-1)/2}}\).
Since the polynomial \(f\) is a product of palindromic unimodal polynomials, it is also palindromic and unimodal.
In particular, the coefficients \(\alpha_m\) are monotonically increasing until the middle (i.e., up to the coefficient \(\alpha_{\floor{(k-1)(q-1)/2}}\)).

Since the sequence \((\delta_1,\dots,\delta_{k-1})\) can be identified with disjoint successive partial sums of the sequence \((\alpha_0,\dots,\alpha_{\floor{(k-1)/2}(q-1)})\) which is monotonically increasing as \(\floor{(k-1)/2}(q-1) \le \floor{(k-1)(q-1)/2}\), it follows that \((\delta_1, \dots, \delta_{k-1})\) is unimodal.
This completes the proof of Theorem~\ref{thm:geometricunimodalbox}.

\section{Number Theoretical Results on Floor and Ceiling Functions}\label{sec:floor}

In this section we prove a variety of results regarding floor and ceiling functions that we use in Section~\ref{sec:limits}.

\begin{proposition}\label{prop:floor-fct}
Let \(k,m,n\) be positive integers and set \(\varepsilon\coloneqq \frac1n\).
Then we have
    \[
        \mleft\lfloor \frac{k\cdot i + 1}{k\cdot m + \varepsilon} \mright\rfloor = \mleft\lfloor \frac{k \cdot i + \delta + 1}{k \cdot m + \varepsilon} \mright\rfloor
    \]
    and
    \[
        \mleft\lceil \frac{k\cdot i + 1}{k\cdot m + \varepsilon} \mright\rceil = \mleft\lceil \frac{k \cdot i + \delta + 1}{k \cdot m + \varepsilon} \mright\rceil
    \]
    for all \(i=0,1,\dots, mn-1\) and all \(\delta=0,1,\dots,k-1\).
\end{proposition}

To prove this statement we consider the function
\[
    f \colon \RR^2 \to \RR; (x,y) \mapsto f(x,y) \coloneqq k\cdot x + y + 1.
\]
We first notice that if we fix \(x_0 \in \RR\) then the function \(f_{x_0}(y) = f(x_0,y)\) is strictly monotonically increasing on \(\RR\).
Similarly, we have for fixed \(y_0 \in \RR\) that the function \(f_{y_0}(x) = f(x,y_0)\) is strictly monotonically increasing on \(\RR\).

\begin{lemma}\label{lem:ineq-restr-fct}
    For each \(\ell \in [0,n-1]\) we have that
    \[
        \ell\cdot (km+\varepsilon) < f(x,y) < (\ell+1) \cdot (km+\varepsilon)
    \]
    where \((x,y) \in [\ell m,(\ell+1)m-1]\times [0,k-1]\).
\end{lemma}

\begin{proof}
Let us consider the restriction of \(f(x,y)\) to the region \([\ell m,(\ell+1)m-1] \times [0,k-1]\) which we denote by \(g(x,y)\).

Notice that the domain of \(g(x,y)\) is a rectangle with left-bottom vertex \((\ell m,0)\) and top-right vertex \(((\ell+1)m-1,k-1)\).
With the above observed monotonicity of \(f(x,y)\), it suffices to show that 
    \[
        \ell \cdot (km + \varepsilon) < f(\ell \cdot m, 0) \qquad \text{and} \qquad f((\ell+1)\cdot m-1,k-1) < (\ell+1)\cdot(km+\varepsilon) \,.
    \]
To prove the first inequality, note that the assumption \(\ell < n\) implies that \(\ell \cdot \varepsilon < 1\), and thus
    \[
		\ell\cdot(km+\varepsilon) = k \cdot \ell \cdot m + \underbrace{\ell \cdot \varepsilon}_{<1} < k \cdot \ell \cdot m +1 = f(\ell \cdot m,0) \,.
    \]
The second inequality is straightforward (and does not need the assumption \(\ell<n\)).
\end{proof}

\begin{proof}[Proof of Proposition~\ref{prop:floor-fct}]
Notice that we can partition the integers in the interval \([0,mn-1]\) into the following pairwise disjoint subsets:
    \[
        [0,mn-1]\cap\ZZ = \bigcup_{\ell=0}^{n-1} \mleft([ \ell m, (\ell+1)m-1] \cap \ZZ\mright) \,.
    \]
Let \(i \in [\ell m, (\ell+1)m-1]\) for some \(\ell = 0, 1, \dots, n-1\) and let \(\delta = 0, 1, \dots, k-1\).
By Lemma~\ref{lem:ineq-restr-fct}, we have that
    \[
        \ell \cdot (k\cdot m+\varepsilon) < k \cdot i+\delta+1 < (\ell+1) \cdot (k\cdot m+\varepsilon),
    \]
    and thus
    \[
        \ell < \frac{k \cdot i + \delta + 1}{k\cdot m+\varepsilon} < \ell+1 \,.
    \]
It follows that
    \[
        \mleft\lfloor \frac{k \cdot i + \delta + 1}{k\cdot m+\varepsilon} \mright\rfloor = \ell \qquad \text{and} \qquad \mleft\lceil \frac{k \cdot i + \delta + 1}{k\cdot m+\varepsilon} \mright\rceil = \ell+1
    \]
for each \(\delta = 0, 1, \dots, k-1\).
\end{proof}

The following statement is straightforward to show:
\begin{lemma}\label{lem:second-ineq-restr-fct}
    For \(\ell \in [0, n-1]\), we have that
    \[
        \ell\cdot (m+\varepsilon) < x + 1 < (\ell+1) \cdot (m+\varepsilon)
    \]
    for all \(x \in [\ell m, (\ell+1)m -1]\).
\end{lemma}

\begin{corollary}\label{cor:relation-betw-k-and-one}
Let \(k, m, n\) be positive integers and set \(\varepsilon\coloneqq \frac1n\).
Then we have
    \[
        \mleft\lfloor \frac{k\cdot q + 1}{k \cdot m + \varepsilon} \mright\rfloor = \mleft\lfloor \frac{q+1}{m+\varepsilon} \mright\rfloor \qquad \text{and} \qquad  \mleft\lceil \frac{k\cdot q + 1}{k \cdot m + \varepsilon} \mright\rceil = \mleft\lceil \frac{q+1}{m+\varepsilon} \mright\rceil
    \]
for all \(q = 0, 1, \dots, mn-1\).
\end{corollary}


\begin{proof}
We again partition the integers in the interval \([0,mn-1]\) into disjoint subsets:
    \[
        \mleft([0,mn-1] \cap \ZZ \mright) = \bigcup_{\ell=0}^{n-1} \mleft( [ \ell m, (\ell+1)m-1] \cap \ZZ \mright) \,.
    \]
Suppose \(q \in [\ell m, (\ell+1)m -1]\) for some \(\ell=0,1,\dots, n-1\).
By the Lemmas~\ref{lem:ineq-restr-fct} and~\ref{lem:second-ineq-restr-fct}, we have that
    \[
        \ell\cdot (km+\varepsilon) < k \cdot q + 1 < (\ell+1) \cdot (km + \varepsilon) \quad \text{and} \quad \ell \cdot (m+\varepsilon) < q + 1 < (\ell+1) \cdot (m+\varepsilon) \,.
    \]
Hence, we get that
    \[
        \ell < \frac{k \cdot q + 1}{km + \varepsilon} < \ell+1 \quad \text{and} \quad \ell < \frac{q+1}{m + \varepsilon} < \ell + 1 \,.
    \]
From this the statement follows.
\end{proof}

\begin{lemma}\label{lem:upper-bound-weighted-floor}
Let \(k,n,m\) be positive integers, and let \(i = 0, 1, \dots, mn-1 \) and \(\delta = 0, 1, \dots, k-1\).
If \(\delta\geq (n-1)m+1\) and \(k\geq n\), then
    \begin{equation}\label{eqn:extensionlemma}
    \mleft\lfloor \frac{ki+\delta+1}{km+\frac{1}{n}}\mright\rfloor \left(1+\frac{1}{k}\right) < \frac{ki+\delta+1}{km+\frac{1}{n}} \, .
    \end{equation}
\end{lemma}

\begin{proof}
Let \(t = 0,1, \dots, n-1\) and \(s = 0, 1, \dots, m-1\).
Observe that since \(\delta\geq (n-1)m+1\), we have 
	\[
		0<nk(ks + \underbrace{\delta - tm}_{\ge1}) \, .
	\]
Since \(0\leq t\leq n-1\) and \(k\geq n\), it follows that
	\[
		0<nk(ks+\delta+1-tm)-t(k+1) = nk (ks + \delta - tm) + \underbrace{nk - t(k+1)}_{\ge1}\, .
	\]
Expanding and rearranging gives us
	\[
		tkmn+t<nk^2s+nk\delta+nk-tk
	\]
from which it follows that
	\[
		\frac{t}{k}<\frac{kns+n\delta+n-t}{kmn+1} \, .
	\]
Adding \(t\) to both sides gives
	\begin{equation}\label{eq:t-equation}
		t\left(1+\frac{1}{k}\right)<t+\frac{kns+n\delta+n-t}{kmn+1} \, .
	\end{equation}
Because \(0\leq s\leq m-1\), \(0 \leq \delta \le k-1\), and \(0\leq t \le n-1\), we have that
	\[
		0\le \frac{kns+n\delta+n-t}{kmn+1}< 1 \quad \Rightarrow \quad t = \left\lfloor t+\frac{kns+n\delta+n-t}{kmn+1}\right\rfloor \, .
	\]
By substituting this expression for \(t\) in the left-hand-side of Equation~\eqref{eq:t-equation}, we obtain
	\[
		\left\lfloor t+\frac{kns+n\delta+n-t}{kmn+1}\right\rfloor\left(1+\frac{1}{k}\right)<t+\frac{kns+n\delta+n-t}{kmn+1} \, .
	\]
Elementary algebraic calculations then yield
	\[
		\left\lfloor \frac{k(tm+s)+\delta+1}{km+\frac{1}{n}}\right\rfloor\left(1+\frac{1}{k}\right)<\frac{k(tm+s)+\delta+1}{km+\frac{1}{n}} \, .
	\]
Note that with our range of \(t\in [0,n-1]\) and \(s\in [0,m-1]\), the values of \(i=tm+s\) exactly parameterize \(i\in [0,mn-1]\), and thus our proof is complete.
\end{proof}

\begin{theorem}\label{thm:ext-floor-ceil-identity}
Let \(k,m,n,r\) be positive integers with \(1\leq r\leq mn\).
For \(k\geq mn\), \\ \(\delta = mn, mn+1, \dots, k-1\), and \(i = 0, 1, \dots, mn-1\), the following hold:
    \begin{equation}\label{eqn:floorextension}
        \left\lfloor \frac{ki+\delta+1}{km+\frac{r}{n}}\right\rfloor = \left\lfloor \frac{ki+\delta+1}{km+\frac{1}{n}}\right\rfloor
    \end{equation}
and
    \begin{equation}\label{eqn:ceilextension}
        \left\lceil \frac{ki+\delta+1}{km+\frac{r}{n}}\right\rceil = \left\lceil \frac{ki+\delta+1}{km+\frac{1}{n}}\right\rceil \, .
    \end{equation} 
\end{theorem}

\begin{proof}
Note that \(1\leq r\) implies that
	\[
		\left\lfloor \frac{ki+\delta+1}{km+\frac{r}{n}}\right\rfloor \leq \left\lfloor \frac{ki+\delta+1}{km+\frac{1}{n}}\right\rfloor
	\]
	and
	\[
		\left\lceil \frac{ki+\delta+1}{km+\frac{r}{n}}\right\rceil \leq \left\lceil \frac{ki+\delta+1}{km+\frac{1}{n}}\right\rceil \, .
	\]
Thus, our result follows once we prove that the following holds for all \(\delta\geq mn\):
	\begin{equation}\label{eqn:extensionineqs}
		\left\lfloor \frac{ki+\delta+1}{km+\frac{1}{n}}\right\rfloor 
		\leq
		\frac{ki+\delta+1}{\left( km+\frac{r}{n}\right) }
		\leq
		\frac{ki+\delta+1}{\left( km+\frac{1}{n}\right) } 
		\leq
		\left\lceil \frac{ki+\delta+1}{km+\frac{r}{n}}\right\rceil \, .
	\end{equation}
Note that the middle inequality above is a consequence of \(1\leq r\).
It is a strict inequality if \(1<r\).
To show the left-most inequality in~\eqref{eqn:extensionineqs}, use \(\delta\geq mn\) and apply Lemma~\ref{lem:upper-bound-weighted-floor} to obtain
	\begin{align*}
		\left\lfloor \frac{ki+\delta+1}{km+\frac{1}{n}}\right\rfloor \left( \frac{km+\frac{r}{n}}{km+\frac{1}{n}}\right) = \left\lfloor \frac{ki+\delta+1}{km+\frac{1}{n}}\right\rfloor \left( 1 + \frac{r-1}{knm+1}\right)
		& <
		\left\lfloor \frac{ki+\delta+1}{km+\frac{1}{n}}\right\rfloor \left(1+\frac{1}{k}\right) \\
		& < \frac{ki+\delta+1}{km+\frac{1}{n}}  \, .
	\end{align*}
From this, we conclude that if \(\delta\geq mn\), the following inequality holds:
	\[
		\left\lfloor \frac{ki+\delta+1}{km+\frac{1}{n}}\right\rfloor 
		<
		\frac{ki+\delta+1}{\left( km+\frac{r}{n}\right) } \, .
	\]
To show the right-most inequality in~\eqref{eqn:extensionineqs}, we assume \(\delta\geq mn\) and consider two cases. 
We write \(i=tm+s\) where \(t = 0, 1, \dots, n-1\) and \(s = 0, 1, \dots, m-1\).

Our first case is when \(s=0\).
Since \(\delta\leq k-1\) and \(m\geq 1\), we have
	\[
		\delta\leq km+\frac{t+1}n-1 = \left(km+\frac{1}{n}\right)(t+1)-ktm-1 \, .
	\]
Combining the above inequality with \(mn \leq \delta \leq k-1\) and \(0\leq t\leq n-1\), we obtain
	\begin{multline*}
		\frac{ki+\delta+1}{km+\frac{1}{n}} =
		\frac{ktm+\delta+1}{km+\frac1n} =
		\frac{ktm+\frac tn + \delta + \frac{n-t}n}{km+\frac1n} =
		t + \frac{n\delta + n - t}{kmn + 1} < t+1 = \\
		= \left\lceil t + \smash{\underbrace{\frac{\delta+1-tm}{(k+1)m}}_{\in(0,1)}}\vphantom{\frac{\delta+1-tm}{(k+1)m}}\right\rceil\vphantom{\underbrace{\frac{\delta+1-tm}{(k+1)m}}_{\in(0,1)}}
		= \left\lceil \frac{ktm+\delta+1}{(k+1)m}\right\rceil
		\leq
		\left\lceil \frac{ktm+\delta+1}{km+\frac{r}{n}}\right\rceil \,.
	\end{multline*}

Our second case is when \(1 \leq s \leq m-1\).
In this case, \(\delta\leq k-1\) implies that
	\[
		\delta\leq k(m-s)-1\leq \frac{t}{n}+\frac{1}{n}+k(m-s)-1 = \left(km+\frac{1}{n}\right)(t+1)-kmt-ks-1
	\]
which implies since \(\delta\geq mn\) and \(i\leq mn-1\) that
	\begin{multline*}
		\frac{ki+\delta+1}{km+\frac{1}{n}} = \frac{k(tm+s)+\delta+1}{km+\frac{1}{n}} < t+1
		= \left\lceil \frac{i}{m}\right\rceil  \\
		< \left\lceil \frac{(k+1)i-i+\delta+1}{km+m}\right\rceil
		\leq \left\lceil \frac{ki+\delta+1}{km+\frac{r}{n}}\right\rceil \, .
	\end{multline*}

This shows that~\eqref{eqn:extensionineqs} holds for all \(mn\leq \delta \leq k-1\).
\end{proof}

\section{Asymptotic Properties of Local \texorpdfstring{\(h^*\)}--Polynomials}\label{sec:limits}

In this section we consider one-row Hermite normal form simplices as in~\eqref{eq:onerow} with fixed \(a_1,\ldots,a_{d-1}\) and increasing normalized volume \(N\).
Specifically, we establish in Theorem~\ref{thm:asymptotic} that the general behavior of the local \(h^*\)-polynomial for these simplices is determined by the simplex with \(N=M+1\), where \(M=\lcm(a_1,\ldots,a_{d-1},-1+\sum_ia_i)\).
The following special case establishes that there is a close connection between the local \(h^*\)-polynomials for \(N=M+1\) and \(N=kM+1\).

\begin{theorem}\label{thm:asymptotic}
	Fix \(a_1, \dots, a_{d-1} \in \ZZ\) and let \(M \coloneqq \mathrm{lcm}(a_1, \dots, a_{d-1}, \sum_{i=1}^{d-1}a_i-1)\).
	Denoting by \(S_{kM+1}\) the simplex given in~\eqref{eq:onerow} with normalized volume \(N=k \cdot M + 1\) for some \(k \in \mathbb{Z}_{>0}\), we have that 
 \[
 B(S_{kM+1};z) = k \cdot B(S_{M+1};z) \, .
 \]
	Thus, \(B(S_{M+1};z)\) is unimodal if and only if \(B(S_{kM+1};z)\) is unimodal for all positive integers \(k\).
\end{theorem}

We first establish Lemma~\ref{lem:asymptoticage}, which is needed for the proof of Theorem~\ref{thm:asymptotic}.
Let \((a_1, \dots, a_{d-1}) \in \mathbb{Z}^{d-1}\) be the vector of the non-trivial-row-entries and let \(N \in \mathbb{Z}_{>0}\) be the normalized volume of the simplex \(S \subseteq \mathbb{R}^d\).
Set \(M \coloneqq \mathrm{lcm}\mleft(a_1, \dots, a_{d-1}, \sum_{i=1}^{d-1}a_i -1\mright) \in \mathbb{Z}_{\ge0}\).
Suppose \(N = N^{(k)} = k \cdot M + 1\) for some \(k \in \mathbb{Z}_{>0}\).
Recall that the parallelepiped group is generated by the following element:
\[
    v_0^{(k)} \coloneqq \mleft( \frac{1-\sum_{i=1}^{d-1}a_i}N, \frac{a_1}N, \frac{a_2}N, \dots, \frac{a_{d-1}}N, -\frac1N \mright) \in \Gamma.
\]
We observe that each element \(x = \ell \cdot v_0^{(k)}\) for \(\ell = 1, \dots, N-1\) has age
\begin{align*}
    \age(x) &= 1 + \mleft\lceil \frac{\ell\cdot \mleft( \mleft(\sum_{i=1}^{d-1} a_i\mright) - 1 \mright)}N\mright\rceil - \sum_{i=1}^{d-1} \mleft\lfloor \frac{\ell \cdot a_i}N \mright\rfloor\\
    &= 1 + \mleft\lceil \frac\ell{N_0}\mright\rceil - \sum_{i=1}^{d-1} \mleft\lfloor \frac\ell{N_i} \mright\rfloor,
\end{align*}
where we have set
\[
    N_0 \coloneqq \frac N{\sum_{i=1}^{d-1}a_i-1} \qquad \text{and} \qquad N_i \coloneqq \frac N{a_i} \quad \text{for \(i=1,\dots,d-1\)}.
\]

We notice that \(N_i\) for \(i=0, \dots, d-1\) is of the form
\[
N_i = \frac{k\cdot M + 1}{D_i} = k \cdot \frac{M}{D_i}+\varepsilon_i = k \cdot M_i + \varepsilon_i
\]
where \(M_i\coloneqq M/D_i\) is an integer and \(\varepsilon_i = 1/D_i\).
Substituting this into the formula for the age, we get
\[
    \age(x) = 1 + \mleft\lceil \frac\ell{k \cdot M_0 + \varepsilon_0} \mright\rceil - \sum_{i=1}^{d-1}\mleft\lfloor \frac\ell{k \cdot M_i + \varepsilon_i}\mright\rfloor.
\]
We now divide \(\ell\) by \(k\), say \(\ell = k \cdot q + \delta + 1\) for some \(\delta=0, 1, \dots, k-1\) and \(q = 0, 1, \dots, M-1\).
By Proposition~\ref{prop:floor-fct}, it follows that
\begin{align*}
    \age\mleft(\ell \cdot v_0^{(k)}\mright) &= 1 + \mleft\lceil \frac{k \cdot q + \delta + 1}{k \cdot M_0 + \varepsilon_0} \mright\rceil - \sum_{i=1}^{d-1}\mleft\lfloor \frac{k \cdot q + \delta + 1}{k \cdot M_i + \varepsilon_i}\mright\rfloor\\
    &= 1 + \mleft\lceil \frac{k \cdot q + 1}{k \cdot M_0 + \varepsilon_0} \mright\rceil - \sum_{i=1}^{d-1}\mleft\lfloor \frac{k \cdot q + 1}{k \cdot M_i + \varepsilon_i}\mright\rfloor = \age\mleft( (k \cdot q + 1) \cdot v_0^{(k)} \mright).
\end{align*}
Here, we used that \(q = 0,1 , \dots, M_i\cdot D_i-1\) for each \(i =1, \dots, d-1\) (notice \(M_i \cdot D_i = M\)).
This proves the following statement.

\begin{lemma}\label{lem:asymptoticage}
    Given the assumptions of Theorem~\ref{thm:asymptotic}, for all \(q = 0, 1, \dots, M-1\) and all \(\delta = 0, 1, \dots, k-1\), we have
    \[
        \age\mleft((k\cdot q+1)v_0^{(k)} \mright) = \age\mleft((k\cdot q + \delta+1)v_0^{(k)}\mright)\,.
    \]
\end{lemma}

\begin{proof}[Proof of Theorem~\ref{thm:asymptotic}]
We want to compare the box polynomials of the two simplices where the \(a_i\) are fixed and only the normalized volume changes, namely the general case \(N_k\) with the `initial case' \(N_1\).
Like before let \(\ell = k \cdot q + \delta + 1\) for some \(\delta = 0,1, \dots, k-1\) and \(q=0,1,\dots,M-1\).
We continue to use the notation from above and notice that by Lemma~\ref{lem:asymptoticage}
\begin{align*}
    \age\mleft((k \cdot q + \delta + 1) v_0^{(k)}\mright) &= \age\mleft((k\cdot q+1)v_0^{(k)}\mright) = 1 + \mleft\lceil \frac{k\cdot q + 1}{k \cdot M_0+\varepsilon_0} \mright\rceil - \sum_{i=1}^{d-1} \mleft\lfloor \frac{k \cdot q+1}{k \cdot M_i + \varepsilon_i} \mright\rfloor\\
    &= 1 + \mleft\lceil \frac{q + 1}{M_0+\varepsilon_0} \mright\rceil - \sum_{i=1}^{d-1} \mleft\lfloor \frac{q+1}{M_i + \varepsilon_i} \mright\rfloor = \age\mleft((q+1) v_0^{(1)}\mright)
\end{align*}
Here, we used Lemma~\ref{cor:relation-betw-k-and-one} to justify the step from the first to the second line.
From this the theorem follows.
\end{proof}
 
Next, we extend Theorem~\ref{thm:asymptotic} to the setting of arbitrary normalized volume.

\begin{theorem}\label{thm:fullasymptotic}
	Fix \(a_1, \dots, a_{d-1} \in \ZZ_{\geq 1}\) and let \(M \coloneqq \mathrm{lcm}(a_1, \dots, a_{d-1}, -1+\sum_{i=1}^{d-1}a_i)\).
	We continue to use the notation from Section~\ref{sec:box-h-star-polynomials} and let \(S_N\) denote the simplex given there with normalized volume \(N\).
	Let \(k\) be a positive integer and \(0\leq r\leq M-1\).
	Then we have that 
    \[
		\lim_{k \to \infty}B(S_{kM+r};z)/B(S_{kM+r};1) = B(S_{M+1};z)/B(S_{M+1};1) \, .
    \]
	It follows that if \(B(S_{M+1};z)\) is strictly unimodal, i.e., if the coefficients are unimodal with strict increases and strict decreases, then \(B(S_{kM+r};z)\) is strictly unimodal for all sufficiently large \(k\).
\end{theorem}

\begin{proof}
	We fix \(r = 0, 1, \dots, M-1\) and set \(N(k,r) = k M + r\) and \(a_0 = \sum_{i=1}^{d-1}a_i-1\).
  Consider the generator of the parallelepiped group from above
	\[
		v_0(k,r) = \mleft( -\frac{a_0}{N(k,r)}, \frac{a_1}{N(k,r)}, \frac{a_2}{N(k,r)}, \dots, \frac{a_{d-1}}{N(k,r)}, -\frac1{N(k,r)} \mright) \in \Gamma \, .
	\]
	For \(\ell=1, 2, \dots, N(k,r)-1\), we write \(\ell = k \cdot q + \delta + 1\) for \(q = 0, 1, \dots, M-1\) and \(\delta = 0, 1, \dots, k-1\), \(M_i \coloneqq M/a_i\) and \(\varepsilon_i \coloneqq r/a_i\) for \(i = 0, 1, \dots, d-1\).
 With this notation, the age of \(\ell\cdot v_0(k,r)\) is:
	\begin{align*}
		\age\mleft(\ell \cdot v_0(k,r)\mright) & = 1 + \mleft\lceil \frac\ell{k \cdot M_0 + \varepsilon_0} \mright\rceil - \sum_{i=1}^{d-1}\mleft\lfloor \frac\ell{k \cdot M_i + \varepsilon_i}\mright\rfloor \\
		& = 1 + \mleft\lceil \frac{k\cdot q + \delta + 1}{k \cdot M_0 + \varepsilon_0} \mright\rceil - \sum_{i=1}^{d-1}\mleft\lfloor \frac{k\cdot q + \delta + 1}{k \cdot M_i + \varepsilon_i}\mright\rfloor
	\end{align*}
	We observe that Theorem~\ref{thm:ext-floor-ceil-identity} shows that
	\[
		\mleft\lfloor \frac{k\cdot q + \delta + 1}{k \cdot M_i + \varepsilon_i} \mright\rfloor = \mleft\lfloor \frac{k \cdot q + \delta + 1}{k \cdot M_i + 1} \mright\rfloor \quad \text{and} \quad \mleft\lceil \frac{k \cdot q+ \delta + 1}{k \cdot M_0 + \varepsilon_0} \mright\rceil = \mleft\lceil \frac{k \cdot q + \delta + 1}{k \cdot M_0 + 1} \mright\rceil
	\]
	for all \(q = 0, 1, \dots, M-1\) and \(\delta = M, M+1, \dots, k-1\) (provided that \(k\) is large enough).

	Hence, it follows for all \(\ell = k \cdot q + \delta + 1\) with \(q = 0, 1, \dots, M-1\) and \(\delta = M, M+1, \dots, k-1\) that
	\[
		\age(\ell \cdot v_0(k,r)) = \age(\ell \cdot v_0(k,1)) \,.
	\]
	In particular, among the \(kM+r-1\) values of \(\ell\) parameterizing the parallelepiped group for \(S_{kM+r}\), the number of \(\ell\)-values with ages differing from \(S_{kM+1}\) is bounded above by \(M^2+r\), which is constant with respect to \(k\).
 In addition to this, in the case that \(\gcd(M,kM+r)>1\), it is possible that some of the \(\ell\)-values for \(S_{kM+r}\) might not yield points in the open box \((0,1)^{d+1}\).
 The maximum possible number of such points is \(\sum_{i=0}^{d-1}a_i\), which is again constant with respect to \(k\).
 
 Thus, as \(k\to \infty\), the number of \(\ell=1,2,\ldots,kM+r-1\) such that \(\age(\ell \cdot v_0(k,r)) = \age(\ell \cdot v_0(k,1))\) is at least \(kM+r-1-\left(M^2+r+\sum_ia_i\right)\).
 Hence, the fraction of such \(\ell\)-values is
 \[
\frac{kM+r-1-\left(M^2+r+\sum_ia_i\right)}{kM+r-1}
 \]
which goes to \(1\) as \(k\to\infty\).
Hence, we have that
	\[
		\lim_{k\to \infty} B(S_{kM+r};z)/B(S_{kM+r};1) = \lim_{k\to \infty}B(S_{kM+1};z)/B(S_{kM+1}; 1)=B(S_{M+1};z)/B(S_{M+1}; 1)\, .
	\]
\end{proof}

For certain values of \(kM+r\), Theorem~\ref{thm:fullasymptotic} has implications for \(h^*\)-polynomials as well.

\begin{corollary}\label{thm:fullasymptotichstar}
Using the notation from Theorem~\ref{thm:fullasymptotic}, if \(\gcd(M,r)=1\), then 
\[
\lim_{k\to \infty} h^*(S_{kM+r};z)/h^*(S_{kM+r};1) = B(S_{M+1};z)/B(S_{M+1}; 1) \, .
\]
Thus, if \(S_{M+1}\) has a strictly unimodal local \(h^*\)-polynomial with a positive linear coefficient, then \(h^*(S_{kM+r};z)\) is unimodal for all sufficiently large \(k\).
\end{corollary}

\begin{proof}
    This is a consequence of Theorems~\ref{thm:boxhstarconditions} and~\ref{thm:fullasymptotic}.
\end{proof}

\section{Further Questions}\label{sec:conclusion}

For a one-row Hermite normal form simplex, Theorem~\ref{thm:fullasymptotic} shows that if the values of \(a_1,\ldots,a_{d-1}\) are fixed, the distribution of the local \(h^*\)-polynomial coefficients is basically determined by the distribution for a single normalized volume of \(M+1\).
This suggests several questions leading to directions for further study.

\begin{figure}
\centering
\includegraphics[width=0.7
\textwidth]{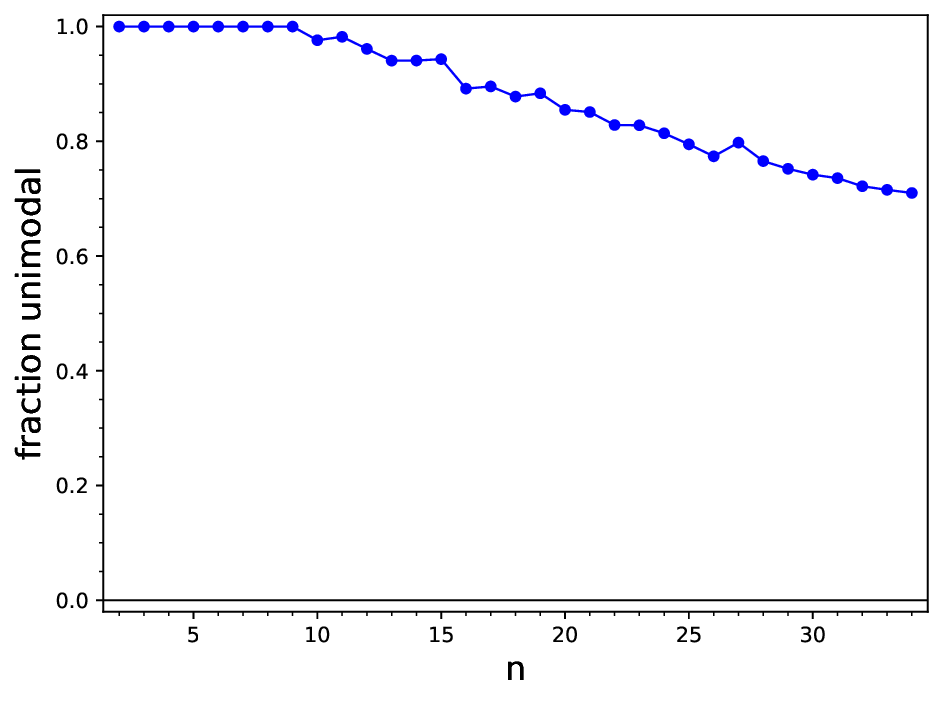}
\caption{For each \(n\), the fraction of unimodal local \(h^*\)-polynomials for one-row Hermite normal form simplices \(S_{M+1}\) with one row given by a partition of \(n\).}
\label{fig:fractionunimodalbypartition}
\end{figure}

\begin{question}\label{q:unimodal}
Which sequences \(a_1,\ldots,a_{d-1}\) yield a unimodal (or strictly unimodal) local \(h^*\)-polynomial for \(S_{M+1}\)?
How common is it for such a simplex to admit a regular unimodular triangulation?
\end{question}

There are several ways to approach Question~\ref{q:unimodal}.
One approach is to fix a positive integer \(n\) and consider the set of all one row Hermite normal form matrices with final row formed by a partition of \(n\) and normalized volume \(M+1\).
One can ask what fraction of these have unimodal local \(h^*\)-polynomials.
A plot of these fraction values is given in Figure~\ref{fig:fractionunimodalbypartition} for \(n\leq 34\).
It is not clear what the long-term behavior of this sequence is.

Another approach, which seems more promising, is motivated by the observation that if the values \(a_1,\ldots,a_{d-1}\) are distinct integers, it appears that this leads to local \(h^*\)-polynomial unimodality.
For example, every partition of \(n\leq 34\) with distinct parts yields a one row Hermite normal form simplex with a unimodal local \(h^*\)-polynomial, leading to the following question.

\begin{question}\label{q:distinct}
If \((a_1,a_2,\ldots,a_{d-1})\) is a list of \(d-1\) distinct positive integers, does the corresponding simplex \(S_{M+1}\) have a unimodal local \(h^*\)-polynomial?
\end{question}

\begin{figure}
\centering
\includegraphics[width=0.7
\textwidth]{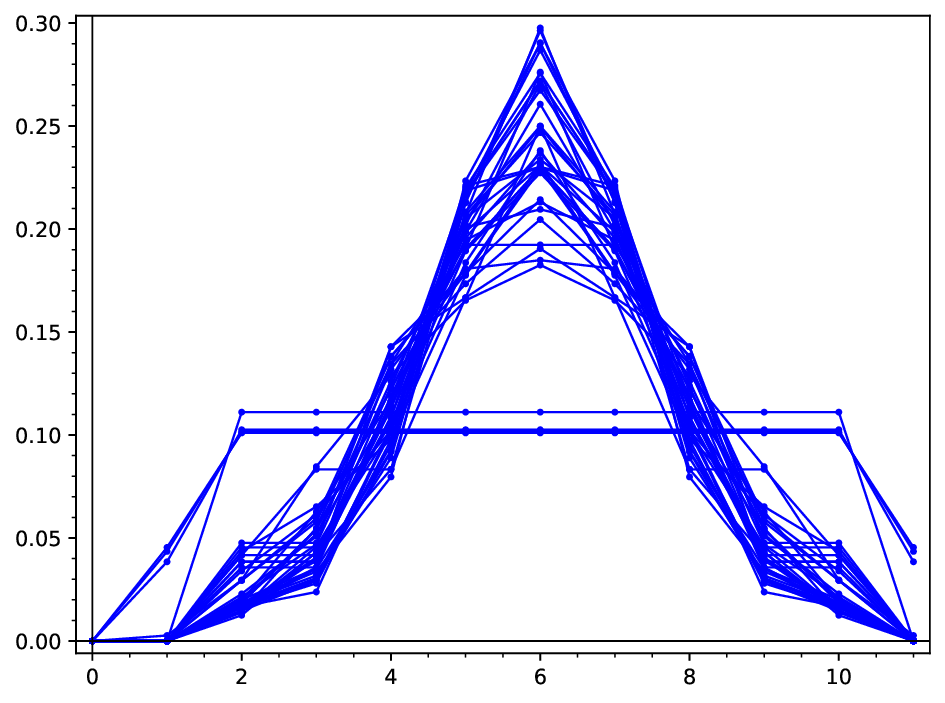}
\caption{The distributions for \(40\) local \(h^*\)-polynomials of \(11\)-dimensional simplices \(S_{M+1}\) generated by small random perturbations from constant rows.}
\label{fig:11randomdist}
\end{figure}

Experimentally, it seems that having a list of completely distinct positive integers is a stronger condition than needed for unimodality.
For example, we generated 121 examples in the following manner.
Begin with a constant vector \(a_i=k\) and fixed \(d\), then randomly add a value from \(\{0,1,2,3,4\}\) to each entry.
We considered the pairs \(d,k\) from 
\[
\{(8,1),(8,4),(8,7),(8,10), (11,1),(11,4),(11,7),(11,10),(14,1),(14,4),(14,7)\}
\]
and generated eleven samples each, consisting of the constant row case and ten random perturbations.
In all cases, we constructed the simplex \(S_{M+1}\), based on its role in Theorem~\ref{thm:fullasymptotic}.
Of this sample, \(95.87\%\) of the \(S_{M+1}\) had unimodal local \(h^*\)-polynomials.
A plot of the distributions of the unimodal local \(h^*\)-polynomials for the \(11\)-dimensional simplices in this sample is given in Figure~\ref{fig:11randomdist}.
The distributions that are constant or nearly constant arise from the constant row values of \(k=1,4,7,10\).
The other distributions arise from the random perturbations, and these all have a pronounced unimodal behavior.

\begin{table}
\begin{tabular}{|l|}
\hline\\
\(a\)-vector\\ local \(h^*\)-vector\\
\hline
\([1, 4, 2, 2, 2, 1, 2, 1, 2, 1]\) \\ \([0, 0, 4, 6, 8, 11, 10, 11, 8, 6, 4, 0]\)\\
\hline
\([2, 3, 3, 2, 3, 4, 3, 4, 3, 3]\) \\ \([0, 0, 12, 27, 54, 57, 48, 57, 54, 27, 12, 0]\)\\
\hline
\([7, 7, 6, 7, 7, 7, 4, 4, 7, 5]\) \\  \([0, 0, 11, 28, 59, 77, 70, 77, 59, 28, 11, 0]\)\\
\hline
\([11, 10, 13, 13, 10, 10, 12, 11, 10, 10]\) \\ \([0, 0, 12738, 45859, 139946, 185372, 167390, 185372, 139946, 45859, 12738, 0]\)\\
\hline
\([6, 5, 4, 7, 6, 5, 4, 4, 4, 4, 4, 6, 7]\) \\ \([0, 0, 84, 126, 213, 533, 888, 886, 886, 888, 533, 213, 126, 84, 0]\)\\ 
\hline
\end{tabular}
\caption{}
\label{table:nonunimodal}
\end{table}

The five non-unimodal examples from our sample are given in Table~\ref{table:nonunimodal}, and each of the corresponding \(a\)-vectors have a single value appearing in around half of the entries.
Note also that unimodality for these examples fails only in the central coefficients.
Thus, it is reasonable to conjecture that if the multiplicity of each distinct entry in the row is sufficiently small relative to the dimension, the local \(h^*\)-polynomial is unimodal.

These observations lead to the following more general question.

\begin{question}\label{q:typical}
For fixed \(d\) and \(N\), consider the set of all one-row Hermite normal form simplices with normalized volume \(N\) and dimension \(d\).
What is the ``typical'' behavior of the local \(h^*\)-polynomial distributions for simplices in this set?
What is the shape of the space of distributions associated to \(B(S;z)\) for all \(S\) in this set?
\end{question}

Regarding Question~\ref{q:typical}, computational experiments suggest that when the values of \\ \(a_1,\ldots,a_{d-1}\) are ``sufficiently'' distinct, the distribution is more similar to Figure~\ref{fig:q3_k12_geometricdist} than Figure~\ref{fig:allonesdist}, as was observed in Figure~\ref{fig:11randomdist}.
However, it is not clear at this time how to translate these observations into a precise conjecture.
Finally, it would be interesting to consider the more general family of Hermite normal form simplices.

\begin{question}\label{q:morerows}
Is there an analogue of Theorem~\ref{thm:fullasymptotic} for Hermite normal form simplices with more than one non-trivial row?
\end{question}

\section*{Data}

The experimental data reported in Section~\ref{sec:conclusion} is available at: \\
\url{https://doi.org/10.17605/OSF.IO/XH58C}.

\section*{Acknowledgments}
The authors thank the anonymous referee for their helpful comments.
The authors thank the American Institute of Mathematics, where this project was started as a result of the workshop ``Ehrhart polynomials: inequalities and extremal constructions.''
The authors also thank Ahmed Umer Ashraf, Matthias Beck, and Marie Meyer for helpful conversations at the start of this project. 
Andr\'es R.~Vindas-Mel\'endez was partially supported by the National Science Foundation under Award DMS-2102921.
Benjamin Braun is partially supported by the National Science Foundation under award DMS-1953785.

\bibliographystyle{plain}
\bibliography{boxpoly}

\begin{bibdiv}
\begin{biblist}

\bib{BajoBeck}{article}{
      author={Bajo, Esme},
      author={Beck, Matthias},
       title={Boundary {$H^*$}-polynomials of rational polytopes},
        date={2023},
        ISSN={0895-4801,1095-7146},
     journal={SIAM J. Discrete Math.},
      volume={37},
      number={3},
       pages={1952\ndash 1969},
         url={https://doi.org/10.1137/22M1508911},
}

\bib{BeckBraunVindas}{article}{
      author={Beck, Matthias},
      author={Braun, Benjamin},
      author={Vindas-Mel\'{e}ndez, Andr\'{e}s~R.},
       title={Decompositions of {E}hrhart {$h^*$}-polynomials for rational
  polytopes},
        date={2022},
        ISSN={0179-5376},
     journal={Discrete Comput. Geom.},
      volume={68},
      number={1},
       pages={50\ndash 71},
         url={https://doi.org/10.1007/s00454-021-00341-0},
}

\bib{BeckLeon}{article}{
      author={Beck, Matthias},
      author={Le\'{o}n, Emerson},
       title={Binomial inequalities for chromatic, flow, and tension
  polynomials},
        date={2021},
        ISSN={0179-5376},
     journal={Discrete Comput. Geom.},
      volume={66},
      number={2},
       pages={464\ndash 474},
         url={https://doi.org/10.1007/s00454-021-00314-3},
}

\bib{ccd}{book}{
      author={Beck, Matthias},
      author={Robins, Sinai},
       title={Computing the continuous discretely},
     edition={Second},
      series={Undergraduate Texts in Mathematics},
   publisher={Springer, New York},
        date={2015},
        ISBN={978-1-4939-2968-9; 978-1-4939-2969-6},
         url={https://doi-org.ezproxy.uky.edu/10.1007/978-1-4939-2969-6},
        note={Integer-point enumeration in polyhedra, With illustrations by
  David Austin},
}

\bib{CRT}{book}{
      author={Beck, Matthias},
      author={Sanyal, Raman},
       title={Combinatorial reciprocity theorems},
      series={Graduate Studies in Mathematics},
   publisher={American Mathematical Society, Providence, RI},
        date={2018},
      volume={195},
        ISBN={978-1-4704-2200-4},
         url={https://doi.org/10.1090/gsm/195},
        note={An invitation to enumerative geometric combinatorics},
}

\bib{BetkeMcMullen}{article}{
      author={Betke, Ulrich},
      author={McMullen, Peter},
       title={Lattice points in lattice polytopes},
        date={1985},
        ISSN={0026-9255},
     journal={Monatsh. Math.},
      volume={99},
      number={4},
       pages={253\ndash 265},
         url={https://doi.org/10.1007/BF01312545},
}

\bib{borgerkretschmernill}{article}{
      author={Borger, Christopher},
      author={Kretschmer, Andreas},
      author={Nill, Benjamin},
       title={Thin polytopes: lattice polytopes with vanishing local
  {$h$}*-polynomial},
        date={2024},
        ISSN={1073-7928,1687-0247},
     journal={Int. Math. Res. Not. IMRN},
      number={7},
       pages={5619\ndash 5657},
         url={https://doi.org/10.1093/imrn/rnad231},
}

\bib{BorisovMavlyutov}{article}{
      author={Borisov, Lev~A.},
      author={Mavlyutov, Anvar~R.},
       title={String cohomology of {C}alabi-{Y}au hypersurfaces via mirror
  symmetry},
        date={2003},
        ISSN={0001-8708},
     journal={Adv. Math.},
      volume={180},
      number={1},
       pages={355\ndash 390},
         url={https://doi.org/10.1016/S0001-8708(03)00007-0},
}

\bib{BrandenJochemko}{article}{
      author={Br\"{a}nd\'{e}n, Petter},
      author={Jochemko, Katharina},
       title={The {E}ulerian transformation},
        date={2022},
        ISSN={0002-9947},
     journal={Trans. Amer. Math. Soc.},
      volume={375},
      number={3},
       pages={1917\ndash 1931},
         url={https://doi.org/10.1090/tran/8539},
}

\bib{braundavisanti}{article}{
      author={Braun, Benjamin},
      author={Davis, Brian},
       title={Antichain simplices},
        date={2020},
     journal={J. Integer Seq.},
      volume={23},
      number={1},
       pages={Art. 20.1.1, 23},
}

\bib{braundavisintclosed}{article}{
      author={Braun, Benjamin},
      author={Davis, Robert},
       title={Ehrhart series, unimodality, and integrally closed reflexive
  polytopes},
        date={2016},
        ISSN={0218-0006},
     journal={Ann. Comb.},
      volume={20},
      number={4},
       pages={705\ndash 717},
         url={https://doi-org.ezproxy.uky.edu/10.1007/s00026-016-0337-6},
}

\bib{braundavissolus}{article}{
      author={Braun, Benjamin},
      author={Davis, Robert},
      author={Solus, Liam},
       title={Detecting the integer decomposition property and {E}hrhart
  unimodality in reflexive simplices},
        date={2018},
        ISSN={0196-8858},
     journal={Adv. in Appl. Math.},
      volume={100},
       pages={122\ndash 142},
         url={https://doi-org.ezproxy.uky.edu/10.1016/j.aam.2018.06.003},
}

\bib{braunhanely}{article}{
      author={Braun, Benjamin},
      author={Hanely, Derek},
       title={Regular unimodular triangulations of reflexive {IDP} 2-supported
  weighted projective space simplices},
        date={2021},
        ISSN={0218-0006},
     journal={Ann. Comb.},
      volume={25},
      number={4},
       pages={935\ndash 960},
         url={https://doi-org.ezproxy.uky.edu/10.1007/s00026-021-00554-3},
}

\bib{braunliu}{article}{
      author={Braun, Benjamin},
      author={Liu, Fu},
       title={{$h$}*-polynomials with roots on the unit circle},
        date={2021},
        ISSN={1058-6458},
     journal={Exp. Math.},
      volume={30},
      number={3},
       pages={332\ndash 348},
         url={https://doi-org.ezproxy.uky.edu/10.1080/10586458.2018.1538912},
}

\bib{conrads}{article}{
      author={Conrads, Heinke},
       title={Weighted projective spaces and reflexive simplices},
        date={2002},
        ISSN={0025-2611},
     journal={Manuscripta Math.},
      volume={107},
      number={2},
       pages={215\ndash 227},
         url={https://doi-org.ezproxy.uky.edu/10.1007/s002290100235},
}

\bib{GKZ91}{article}{
      author={Gel'fand, Izrail~M.},
      author={Kapranov, Mikhail~M.},
      author={Zelevinsky, Andrey~V.},
       title={Discriminants of polynomials in several variables and
  triangulations of {N}ewton polyhedra},
        date={1990},
        ISSN={0234-0852},
     journal={Algebra i Analiz},
      volume={2},
      number={3},
       pages={1\ndash 62},
}

\bib{solusgustafsson}{article}{
      author={Gustafsson, Nils},
      author={Solus, Liam},
       title={Derangements, {E}hrhart theory, and local {$h$}-polynomials},
        date={2020},
        ISSN={0001-8708},
     journal={Adv. Math.},
      volume={369},
       pages={107169, 35},
         url={https://doi-org.ezproxy.uky.edu/10.1016/j.aim.2020.107169},
}

\bib{hibihermite}{article}{
      author={Hibi, Takayuki},
      author={Higashitani, Akihiro},
      author={Li, Nan},
       title={Hermite normal forms and {$\delta$}-vectors},
        date={2012},
        ISSN={0097-3165},
     journal={J. Combin. Theory Ser. A},
      volume={119},
      number={6},
       pages={1158\ndash 1173},
         url={https://doi-org.ezproxy.uky.edu/10.1016/j.jcta.2012.02.005},
}

\bib{hibitsuchiyayoshida}{article}{
      author={Hibi, Takayuki},
      author={Tsuchiya, Akiyoshi},
      author={Yoshida, Koutarou},
       title={Gorenstein simplices with a given {$\delta$}-polynomial},
        date={2019},
        ISSN={0012-365X},
     journal={Discrete Math.},
      volume={342},
      number={12},
       pages={111619, 10},
         url={https://doi-org.ezproxy.uky.edu/10.1016/j.disc.2019.111619},
}

\bib{higashitanishifted}{article}{
      author={Higashitani, Akihiro},
       title={Shifted symmetric {$\delta$}-vectors of convex polytopes},
        date={2010},
        ISSN={0012-365X},
     journal={Discrete Math.},
      volume={310},
      number={21},
       pages={2925\ndash 2934},
         url={https://doi-org.ezproxy.uky.edu/10.1016/j.disc.2010.06.044},
}

\bib{higashitanicounterexamples}{article}{
      author={Higashitani, Akihiro},
       title={Counterexamples of the conjecture on roots of {E}hrhart
  polynomials},
        date={2012},
        ISSN={0179-5376},
     journal={Discrete Comput. Geom.},
      volume={47},
      number={3},
       pages={618\ndash 623},
         url={https://doi-org.ezproxy.uky.edu/10.1007/s00454-011-9390-4},
}

\bib{higashitaniprime}{article}{
      author={Higashitani, Akihiro},
       title={Ehrhart polynomials of integral simplices with prime volumes},
        date={2014},
     journal={Integers},
      volume={14},
       pages={Paper No. A45, 15},
}

\bib{Jochemko}{article}{
      author={Jochemko, Katharina},
       title={Symmetric decompositions and the {V}eronese construction},
        date={2022},
        ISSN={1073-7928},
     journal={Int. Math. Res. Not. IMRN},
      number={15},
       pages={11427\ndash 11447},
         url={https://doi.org/10.1093/imrn/rnab031},
}

\bib{Leon}{article}{
      author={Le\'{o}n, Emerson},
       title={Stapledon decompositions and inequalities for coefficients of
  chromatic polynomials},
        date={2017},
     journal={S\'{e}m. Lothar. Combin.},
      volume={78B},
       pages={Art. 24, 12},
         url={https://doi.org/10.1155/2017/1084769},
}

\bib{liusolus}{article}{
      author={Liu, Fu},
      author={Solus, Liam},
       title={On the relationship between {E}hrhart unimodality and {E}hrhart
  positivity},
        date={2019},
        ISSN={0218-0006},
     journal={Ann. Comb.},
      volume={23},
      number={2},
       pages={347\ndash 365},
         url={https://doi-org.ezproxy.uky.edu/10.1007/s00026-019-00429-8},
}

\bib{nillschepers}{article}{
      author={Nill, Benjamin},
      author={Schepers, Jan},
       title={Combinatorial questions related to stringy {E}-polynomials of
  {G}orenstein polytopes},
        date={2012},
     journal={Oberwolfach {R}eports},
      volume={62},
       pages={62\ndash 64},
}

\bib{payne}{article}{
      author={Payne, Sam},
       title={Ehrhart series and lattice triangulations},
        date={2008},
        ISSN={0179-5376},
     journal={Discrete Comput. Geom.},
      volume={40},
      number={3},
       pages={365\ndash 376},
         url={https://doi-org.ezproxy.uky.edu/10.1007/s00454-007-9002-5},
}

\bib{svl}{article}{
      author={Schepers, Jan},
      author={Van~Langenhoven, Leen},
       title={Unimodality questions for integrally closed lattice polytopes},
        date={2013},
        ISSN={0218-0006},
     journal={Ann. Comb.},
      volume={17},
      number={3},
       pages={571\ndash 589},
         url={https://doi-org.ezproxy.uky.edu/10.1007/s00026-013-0185-6},
}

\bib{schrijver}{book}{
      author={Schrijver, Alexander},
       title={Theory of linear and integer programming},
      series={Wiley-Interscience Series in Discrete Mathematics},
   publisher={John Wiley \& Sons, Ltd., Chichester},
        date={1986},
        ISBN={0-471-90854-1},
        note={A Wiley-Interscience Publication},
}

\bib{soluslocal}{incollection}{
      author={Solus, Liam},
       title={Local {$h^*$}-polynomials of some weighted projective spaces},
        date={2019},
   booktitle={Algebraic and geometric combinatorics on lattice polytopes},
   publisher={World Sci. Publ., Hackensack, NJ},
       pages={382\ndash 399},
}

\bib{solusnumeral}{article}{
      author={Solus, Liam},
       title={Simplices for numeral systems},
        date={2019},
        ISSN={0002-9947},
     journal={Trans. Amer. Math. Soc.},
      volume={371},
      number={3},
       pages={2089\ndash 2107},
         url={https://doi-org.ezproxy.uky.edu/10.1090/tran/7424},
}

\bib{StanleySubdivisions}{article}{
      author={Stanley, Richard~P.},
       title={Subdivisions and local {$h$}-vectors},
        date={1992},
        ISSN={0894-0347},
     journal={J. Amer. Math. Soc.},
      volume={5},
      number={4},
       pages={805\ndash 851},
         url={https://doi.org/10.2307/2152711},
}

\bib{stapledonineq}{article}{
      author={Stapledon, Alan},
       title={Inequalities and {E}hrhart {$\delta$}-vectors},
        date={2009},
        ISSN={0002-9947},
     journal={Trans. Amer. Math. Soc.},
      volume={361},
      number={10},
       pages={5615\ndash 5626},
         url={https://doi-org.ezproxy.uky.edu/10.1090/S0002-9947-09-04776-X},
}

\bib{sage}{manual}{
      author={{The Sage Developers}},
       title={{S}agemath, the {S}age {M}athematics {S}oftware {S}ystem
  ({V}ersion 9.3)},
        date={2021},
        note={{\tt https://www.sagemath.org}},
}

\bib{tsuchiyagorenstein}{article}{
      author={Tsuchiya, Akiyoshi},
       title={Gorenstein simplices and the associated finite abelian groups},
        date={2018},
        ISSN={0195-6698},
     journal={European J. Combin.},
      volume={67},
       pages={145\ndash 157},
         url={https://doi-org.ezproxy.uky.edu/10.1016/j.ejc.2017.07.022},
}

\end{biblist}
\end{bibdiv}

\end{document}